\documentclass[leqno,11pt]{amsart}
\usepackage{amssymb}
\usepackage{rotating}
\newcommand{\Q}{{\mathbb{Q}}}

\newcommand{\C}{{\mathbb{C}}}
\newcommand{\Z}{{\mathbb{Z}}}
\newcommand{\R}{{\mathbb{R}}}
\newcommand{\T}{{\mathbb{T}}}

\newcommand{\tc}{{\underline{c}}}
\newcommand{\ba}{{\boldsymbol{a}}}
\newcommand{\bc}{{\boldsymbol{c}}}

\newcommand{\bD}{{\mathbf{D}}}

\newcommand{\bC}{{\mathbf{C}}}
\newcommand{\bH}{{\mathbf{H}}}
\newcommand{\bJ}{{\mathbf{J}}}

\newcommand{\fp}{{\mathfrak{p}}}
\newcommand{\fs}{{\mathfrak{s}}}
\newcommand{\ft}{{\mathfrak{t}}}
\newcommand{\fu}{{\mathfrak{u}}}

\newcommand{\fC}{{\mathfrak{C}}}
\newcommand{\fF}{{\mathcal{F}}}
\newcommand{\fL}{{\mathfrak{L}}}

\newcommand{\cA}{{\mathcal{A}}}

\newcommand{\cE}{{\mathcal{E}}}

\newcommand{\cD}{{\mathcal{D}}}
\newcommand{\cLR}{{\mathcal{LR}}}

\newcommand{\cO}{{\mathcal{O}}}

\newcommand\Irr{\operatorname{Irr}}

\newtheorem{thm}{Theorem}[section]
\newtheorem{cor}[thm]{Corollary}
\newtheorem{prop}[thm]{Proposition}
\newtheorem{lem}[thm]{Lemma}
\newtheorem{conj}[thm]{Conjecture}

\theoremstyle{definition}
\newtheorem{defn}[thm]{Definition}
\newtheorem{exmp}[thm]{Example}

\theoremstyle{remark}
\newtheorem{rem}[thm]{Remark}

\renewcommand{\leq}{\leqslant}
\renewcommand{\geq}{\geqslant}
\renewcommand{\atop}[2]{\genfrac{}{}{0pt}{}{#1}{#2}}

\address{M.G.: Department of Mathematical Sciences, King's College, 
Aberdeen AB24 3UE, Scotland, UK}

\email{geck@maths.abdn.ac.uk}

\begin{document}

\date{}

\title{Leading coefficients and cellular bases of Hecke algebras}

\author{Meinolf Geck}


\begin{abstract} 
Let $\bH$ be the generic Iwahori--Hecke algebra associated with a finite
Coxeter group $W$. Recently, we have shown that $\bH$ admits a natural 
cellular basis in the sense of Graham--Lehrer, provided that $W$ is a Weyl 
group and all parameters of $\bH$ are equal. The construction involves some 
data arising from the Kazhdan--Lusztig basis $\{\bC_w\}$ of $\bH$ and 
Lusztig's asymptotic ring $\bJ$. This article attemps to study $\bJ$ and its 
representation theory from a new point of view. We show that $\bJ$ can be 
obtained in an entirely different fashion from the generic representations 
of $\bH$, without any reference to $\{\bC_w\}$. Then we can extend the 
construction of the cellular basis to the case where $W$ is not 
crystallographic. Furthermore, if $\bH$ is a multi-parameter algebra, we 
will see that there always exists at least one cellular structure on $\bH$. 
Finally, one may also hope that the new construction of $\bJ$ can be 
extended to Hecke algebras associated to complex reflection groups.
\end{abstract}

\maketitle

\pagestyle{myheadings}
\markboth{Geck}{Leading coefficients and cellular bases}

\section{Introduction} \label{sec0}

Let $\bH$ be a generic $1$-parameter Iwahori--Hecke algebra associated to a 
finite Weyl group $W$, defined over a suitable ring of Laurent polynomials. 
(More precise definitions will be given below.) By definition, $\bH$ has a 
standard basis usually denoted by $\{T_w\mid w\in W\}$. Using properties of 
the ``new'' basis $\{\bC_w\mid w \in W\}$ introduced in \cite{KaLu}, 
Lusztig has defined a ring $\bJ$ which has a $\Z$-basis $\{t_w \mid w 
\in W\}$ and integral structure constants, and which can be viewed as an 
``asymptotic'' version of $\bH$. All the ingredients in the construction 
of $\bJ$ can be defined in an elementary way, but the proof that we indeed 
obtain an associative ring with identity requires a deep geometric 
interpretation of the basis $\{\bC_w\}$; see \cite{Lu2}, \cite{Lusztig03}. 

It turns out that $\bJ_\Q=\Q \otimes_{\Z} \bJ$ is a split semisimple algebra 
isomorphic to the group algebra of $W$. Using properties of the irreducible 
representations of $\bJ_\Q$, we have recently proved in \cite{mycell} that 
$\bH$ has a natural ``cellular'' structure in the sense of Graham and Lehrer 
\cite{GrLe}. The elements of the  corresponding ``cellular'' basis of $\bH$ 
are certain $\Z$-linear combinations of the basis $\{\bC_w\}$ where the 
coefficients involve data arising from the action of the basis elements 
$t_w$ in the irreducible representations of $\bJ_\Q$. Note that, although 
there is an isomorphism between $\bJ_\Q$ and the group algebra of $W$, it 
does not seem to be easily possible to see the data that we need through 
this isomorphism. (For example, the image of $t_w$ in the group algebra of 
$W$ is, in general, a rather complicated sum of group elements.)

Now Lusztig \cite{Lusztig83}, \cite{Lusztig03} has shown that the 
construction of $\bJ$ also makes sense---under the assumption that the 
conjectures {\bf P1}--{\bf P15} in \cite[14.2]{Lusztig03} hold---when we 
consider an Iwahori--Hecke algebra $\bH$ with possibly unequal parameters. 
The results in \cite{mycell} also extend to this case, assuming that
{\bf P1}--{\bf P15} hold.

One of the purposes of this paper is to show that the data required
to define a ``cellular'' basis of $\bH$ can be obtained in an alternative
way, using the generic irreducible representations of $\bH$ and the leading 
matrix coefficients introduced in \cite{my02}. These coefficients even allow
us to construct a ring $\tilde{\bJ}$ with rational structure constants, 
and show that it is associative with identity, without any reference to 
the Kazhan--Lusztig basis $\{\bC_w\}$ at all. We expect that we have 
$\bJ=\tilde{\bJ}$ in general but, at present, we can only prove this 
equality assuming that Lusztig's conjectures {\bf P1}--{\bf P15} hold.

As an application, we extend the construction of a ``cellular'' basis to 
Iwahori--Hecke algebras associated with non-crystallographic finite Coxeter 
groups, as announced in \cite[Remark~3.3]{mycell}.  Using the results in 
\cite{klremarks}, we can also show that an Iwahori--Hecke algebra with 
possibly unequal parameters always admits at least one ``cellular'' 
structure.

Another aspect of our construction of the ring $\tilde{\bJ}$ is that it
may actually be applied to other types of algebras, like the cyclotomic 
Hecke algebras of Brou\'e--Malle \cite{BrMa93} associated to complex 
reflection groups. We hope to discuss this in more detail elsewhere.

This paper is organised as follows. In Section~2, we briefly recall the main
facts about the Kazhdan--Lusztig basis and the $\ba$-invariants of the
irreducible representations of $W$. Here, we work in the general case of
possibly unequal parameters. In Proposition~\ref{klrem1}, we recall a 
result from \cite{klremarks} which shows that the structure constants of 
Lusztig's ring $\bJ$ can be expressed in terms of the leading matrix 
coefficients'' of \cite{my02}. This is the starting point for our 
construction of a new ring $\tilde{\bJ}$; see Section~3. For this purpose, 
we use a definition of the leading matrix coefficients which is somewhat 
more general than that in \cite{my02}; this generalisation is necessary to 
obtain the strongest possible statements in our applications. The new 
definition involves the concept of ``balanced'' representations which will 
be studied in more detail in Section~4. In particular, we establish an 
efficient criterion for checking if a given representation is balanced or 
not; see Proposition~\ref{bal2}. We will show that the analogue of 
\cite[Prop.~2.6]{mycell} (which describes the data required to define a 
cellular basis) holds for all types of $W$ and all choices of the parameters.
In Section~5, we formulate the hypothesis {\bf $\widetilde{\mbox{P15}}$}
which is a variant of Lusztig's {\bf P15} in \cite[14.2]{Lusztig03}. This 
hypothesis alone allows us to construct a cellular basis of $\bH$; the 
statement in Theorem~\ref{mainthm} is actually slightly stronger than the 
main result of \cite{mycell}. In the process of doing this, we give a 
simplified treatment of Lusztig's homomorphism from $\bH$ into $\bJ$; see 
Theorem~\ref{thmJ}.

Let us now introduce some basic notation that will be used throughout this
paper.  Let $(W,S)$ be a Coxeter system and $l\colon W \rightarrow 
\Z_{\geq 0}$ be the usual length function. In this paper, we will only
consider the case where $W$ is a finite group. Let $\Gamma$ be an abelian 
group (written additively). Following Lusztig \cite{Lusztig03}, a function 
$L \colon W \rightarrow \Gamma$ is called a {\em weight function} if $L(ww')
=L(w)+L(w')$ whenever $w,w'\in W$ are such that $l(ww')=l(w)+l(w')$. Note 
that $L$ is uniquely determined by the values $\{L(s)\mid s \in S\}$. 
Furthermore, if $\{c_s \mid s \in S\}$ is a collection of elements in 
$\Gamma$ such that $c_s=c_t$ whenever $s,t \in S$ are conjugate in $W$, then 
there is (unique) weight function $L\colon W \rightarrow \Gamma$ such that 
$L(s)=c_s$ for all $s \in S$. 

Let $R\subseteq \C$ be a subring and $A=R[\Gamma]$ be the free $R$-module 
with basis $\{\varepsilon^g \mid g\in \Gamma\}$. There is a well-defined 
ring structure on $A$ such that $\varepsilon^g\varepsilon^{g'}=
\varepsilon^{g+g'}$ for all $g,g' \in \Gamma$. We write $1=\varepsilon^0 
\in A$. Given $a\in A$ we denote by $a_g$ the coefficient of 
$\varepsilon^g$, so that $a=\sum_{g\in \Gamma} a_g\,\varepsilon^g$.
Let $\bH=\bH_A(W,S,L)$ be the {\em generic Iwahori--Hecke algebra} over $A$
with parameters $\{v_s \mid s\in S\}$ where $v_s:=\varepsilon^{L(s)}$ for
$s\in S$. This an associative algebra which is free as an $A$-module, with
basis $\{T_w\mid w \in W\}$. The multiplication is given by the rule
\[ T_sT_w=\left\{\begin{array}{cl} T_{sw} & \quad \mbox{if $l(sw)>l(w)$},\\
T_{sw}+(v_s-v_s^{-1})T_w & \quad \mbox{if $l(sw)<l(w)$},\end{array}
\right.\]
where $s\in S$ and $w\in W$. The element $T_1$ is the identity element.

\begin{exmp} \label{Mweightint} Assume that $\Gamma=\Z$. Then $A$ is
nothing but the ring of Laurent polynomials over $R$ in an
indeterminate~$\varepsilon$; we will usually denote $v=\varepsilon$. 
Then $\bH$ is an associative algebra over $A=R[v,v^{-1}]$ with
relations:
\[ T_sT_w=\left\{\begin{array}{cl} T_{sw} & \quad \mbox{if $l(sw)>l(w)$},\\
T_{sw}+(v^{c_s}-v^{-c_s})T_w & \quad \mbox{if $l(sw)<l(w)$},\end{array}
\right.\]
where $s\in S$ and $w\in W$.  
This is the setting of Lusztig \cite{Lusztig03}.
\end{exmp}

\begin{exmp} \label{Masym} (a) Assume that $L$ is constant $S$; this case
will be referred to as the {\em equal parameter case}. Note that we are 
automatically in this case when $W$ is of type $A_{n-1}$, $D_n$, $I_2(m)$ 
where $m$ is odd, $H_3$, $H_4$, $E_6$, $E_7$ or $E_8$ (since all generators
in $S$ are conjugate in $W$).

(b) Assume that $W$ is irreducible. Then unequal parameters can only arise 
in types $B_n$, $F_4$, and $I_2(m)$ where $m$ is even.
\end{exmp}

\begin{exmp} \label{Mrem12}
A ``universal'' weight function is given as follows. Let $\Gamma_0$ be the
group of all tuples $(n_s)_{s \in S}$ where $n_s \in \Z$ for all $s \in S$
and $n_s=n_t$ whenever $s,t\in S$ are conjugate in $W$.  (The addition is
defined componentwise). Let $L_0\colon W \rightarrow \Gamma_0$ be the
weight function given by sending $s\in S$ to the tuple $(n_t)_{t \in S}$
where $n_t=1$ if $t$ is conjugate to $s$ and $n_t=0$, otherwise. Let 
$A_0=R[\Gamma_0]$ and $\bH_0=\bH_{A_0}(W,S,L_0)$ be the associated 
Iwahori--Hecke algebra, with parameters $\{v_s\mid s \in S\}$. Then $A_0=
R[\Gamma_0]$ is nothing but the ring of Laurent polynomials in 
indeterminates $v_s$ ($s\in S$) with coefficients in $R$, where $v_s=v_t$ 
whenever $s,t\in S$ are conjugate in $W$. Furthermore, if $S'\subseteq S$ 
is a set of representatives for the classes of $S$ under conjugation, then 
$\{v_s \mid s \in S'\}$ are algebraically independent.
\end{exmp}

\section{The Kazhdan--Lusztig basis and leading matrix coefficients} 
\label{sec0a}

We now introduce two concepts whose interplay is the main subject of
this paper: the Kazhdan--Lusztig basis and leading matrix coefficients.
Both of these essentially rely on the choice of a total ordering $\leq$ 
on $\Gamma$ which is compatible with the group structure, that is, whenever
$g,g',h \in \Gamma$ are such that $g\leq g'$, then $g+h\leq g'+h$. Such an 
order on $\Gamma$ will be called a {\em monomial order}. 

We will assume that such an ordering exists on $\Gamma$. One readily 
checks that this implies that $A=R[\Gamma]$ is an integral domain; we 
usually reserve the letter $K$ to denote its field of fractions. If we 
are in the equal parameter case (Example~\ref{Masym}), the group $\Gamma=
\Z$ has a natural monomial order. On the other hand, in the setting of 
Example~\ref{Mrem12} (assuming that not all elements of $S$ are conjugate), 
there are infinitely many monomial orders on $\Gamma$.

Throughout this paper, we fix a choice of a monomial order, and we
assume that  
\[ L(s)>0 \qquad \mbox{for all $s \in S$}.\]
We define $\Gamma_{\geq 0}=\{g\in \Gamma\mid g\geq 0\}$ and denote by
$\Z[\Gamma_{\geq 0}]$ the set of all integral linear combinations of
terms $\varepsilon^g$ where $g\geq 0$. The notations $\Z[\Gamma_{>0}]$, 
$\Z[\Gamma_{\leq 0}]$, $\Z[\Gamma_{<0}]$ have a similar meaning.

\subsection{The $\ba$-invariants} \label{sub21}
We set $\Z_W:=\Z[2\cos(2\pi/m_{st}) \mid s,t \in S]$ (where $m_{st}$ 
denotes the order of $st$ in $W$). Note that $\Z_W=\Z$ if $W$ is a finite 
Weyl group (or of crystallographic type), that is, if $m_{st} \in \{2,3,
4,6\}$ for all $s,t \in S$. Recall that $R$ is a subring of $\C$. We shall 
always assume that 

\smallskip
\begin{center}
\fbox{$\qquad \Z_W \subseteq R\quad$ and $\quad F$ is the field of fractions
of $R$.}
\end{center}
\smallskip

Then it is known that $F$ is a splitting field for $W$; see 
\cite[Theorem~6.3.8]{gepf}. The set of irreducible representations of $W$ 
(up to isomorphism) will be denoted by
\[ \Irr(W)=\{E^\lambda \mid \lambda \in \Lambda\}\]
where $\Lambda$ is some finite indexing set and $E^\lambda$ is an
$F$-vectorspace with a given $F[W]$-module structure.  We shall also write
\[d_\lambda=\dim E^\lambda \qquad \mbox{for all $\lambda \in \Lambda$}.\]
Let $K$ be the field of fractions of $A$. By 
extension of scalars, we obtain a $K$-algebra $\bH_K=K\otimes_A \bH$. This 
algebra is known to be split semisimple; see \cite[9.3.5]{gepf}. Furthermore,
by Tits' Deformation Theorem, the irreducible representations of $\bH_K$ 
(up to isomorphism) are in bijection with the irreducible representations 
of $W$; see \cite[8.1.7]{gepf}. Thus, we can write
\[ \Irr(\bH_K)=\{E^\lambda_\varepsilon \mid \lambda \in \Lambda\}.\]
The correspondence $E^\lambda \leftrightarrow E^\lambda_\varepsilon$ is 
uniquely determined by the following condition:
\[ \mbox{trace}\bigl(w,E^\lambda\bigr)=\theta_1\bigl(\mbox{trace}(T_w,
E^\lambda_\varepsilon)\bigr) \qquad \mbox{for all $w \in W$},\]
where $\theta_1 \colon A \rightarrow R$ is the unique ring homomorphism
such that $\theta_1(\varepsilon^g)=1$ for all $g \in \Gamma$. Note also 
that $\mbox{trace}\bigl(T_w,E^\lambda_\varepsilon\bigr) \in A$ for 
all $w\in W$. 

The algebra $\bH$ is {\em symmetric}, with trace from $\tau \colon \bH 
\rightarrow A$ given by $\tau(T_1)=1$ and $\tau(T_w)=0$ for $1 \neq w 
\in  W$. The sets $\{T_w \mid w \in W\}$ and $\{T_{w^{-1}}\mid w \in W\}$ 
form a pair of dual bases. Hence we have the following orthogonality 
relations:
\[ \sum_{w \in W} \mbox{trace}\bigl(T_w,E^\lambda_\varepsilon\bigr)
\,\mbox{trace}\bigl(T_{w^{-1}},E_\varepsilon^\mu\bigr)=\left
\{\begin{array}{cl} d_\lambda\,\bc_\lambda & \quad \mbox{if $\lambda=\mu$},
\\ 0 & \quad \mbox{if $\lambda \neq\mu$};\end{array}\right.\]
see \cite[8.1.7]{gepf}. Here, $0 \neq \bc_\lambda \in A$ and,
following Lusztig, we can write

\smallskip
\begin{center}
\fbox{$\displaystyle \bc_\lambda=f_\lambda\, \varepsilon^{-2\ba_\lambda}+
\mbox{combination of terms $\varepsilon^g$ where $g>-2\ba_\lambda$},$}
\end{center}
\smallskip
where $\ba_\lambda \in \Gamma_{\geq 0}$ and $f_\lambda$ is a strictly 
positive real number; see \cite[3.3]{my02}. 

\begin{rem} \label{ainv} The invariants $\ba_\lambda$ and $f_\lambda$ are 
explicitly known for all types of $W$; see Lusztig \cite[Chap.~22]{Lusztig03}. 
The elements $\bc_\lambda \in A$ and the coefficients $f_\lambda$ are 
independent of the monomial order $\leq$, but $\ba_\lambda$ heavily depends 
on it. Note that the statement concerning the independence of $f_\lambda$ 
is of interest only in the unequal parameter case; see \cite[Prop.~5.1 and 
Table~1]{klremarks} for types $F_4$ and $I_2(m)$, and 
\cite[Prop.~22.14]{Lusztig03} for type $B_n$.
\end{rem}

The invariants $\ba_\lambda$ play a fundamental role in Lusztig's study 
\cite{LuBook} of the characters of reductive groups over finite fields. In 
\cite{mycell}, we use these invariants to define an ordering of $\Lambda$,
which is an essential ingredient in the construction of a ``cellular'' 
basis of $\bH$. 

\subsection{Balanced representations} \label{sub22}
We can now introduce the notion of ``balanced'' representations, which 
is slightly more general than the related concept of ``orthogonal'' 
representations introduced in \cite{my02}. For this purpose, following 
\cite{my02}, we consider a certain valuation ring $\cO$ in $K$. Let us write 
\begin{align*}
F[\Gamma_{\geq 0}]=\mbox{set of $F$-linear combinations of terms 
$\varepsilon^g$ where $g\geq 0$},\\
F[\Gamma_{>0}]=\mbox{set of $F$-linear combinations of terms 
$\varepsilon^g$ where $g>0$}.
\end{align*}
Note that $1+F[\Gamma_{>0}]$ is multiplicatively closed. Furthermore, every 
element $x\in K$ can be written in the form
\[ x=r_x\,\varepsilon^{g_x}\frac{1+p}{1+q}\qquad \mbox{where $r_x 
\in F$, $g_x \in \Gamma$ and $p,q\in F[\Gamma_{>0}]$};\]
note that, if $x\neq 0$, then $r_x$ and $g_x$ indeed are
{\em uniquely determined} by $x$; if $x=0$, we have $r_0=0$ and we set
$g_0:=+\infty$ by convention. We set
\[{\cO}:=\{x \in K \mid g_x \geq 0\} \qquad \mbox{and}\qquad
{\fp}:=\{x \in K \mid g_x >0\}.\]
Then it is easily verified that $\cO$ is a valuation ring in $K$, with
maximal ideal $\fp$. Note that we have
\[ \cO \cap F[\Gamma]=F[\Gamma_{\geq 0}] \qquad \mbox{and}\qquad 
\fp\cap F[\Gamma]=F[\Gamma_{>0}].\]
We have a well-defined $F$-linear ring homomorphism $\cO \rightarrow F$
with kernel $\fp$. The image of $x\in \cO$ in $F$ is called the
{\em constant term} of $x$. Thus, the constant term of $x$ is $0$ if
$x\in \fp$; the constant term equals $r_x$ if $x\in \cO^\times$.

\begin{defn} \label{bal} Choosing a basis of $E_\varepsilon^\lambda$, we 
obtain a matrix representation $\rho^\lambda \colon \bH_K 
\rightarrow M_{d_\lambda}(K)$.  Given $h\in \bH_K$ and $1\leq i,j\leq
d_\lambda$, we denote by $\rho^\lambda_{ij}(h)$ the $(i,j)$-entry of 
the matrix $\rho^\lambda(h)$. We say that $\rho^\lambda$ is {\em balanced}
if \[ \varepsilon^{\ba_\lambda}\rho^\lambda_{ij}(T_w) \in \cO \qquad
\mbox{for all $w \in W$ and all $i,j \in\{1,\ldots,d_\lambda\}$}.\]
If $\rho^\lambda$ is balanced, we define the {\em leading matrix
coefficient} $c_{w,\lambda}^{ij} \in F$ to be the constant term of 
$(-1)^{l(w)} \varepsilon^{\ba_\lambda} \rho^\lambda_{ij}(T_w)$.
\end{defn}

\begin{prop}[Cf.\ \protect{\cite[\S 4]{my02}}] \label{balex} For each 
$\lambda \in \Lambda$, there exists a balanced representation 
$\rho^\lambda$ afforded by $E^\lambda_{\varepsilon}$; moreover, 
$\rho^\lambda$ can be chosen such that
\[\Delta^\lambda \,\rho^\lambda (T_{w^{-1}})=\rho^\lambda
(T_w)^{\operatorname{tr}} \,\Delta^\lambda \quad \mbox{for all $w \in W$},\]
where $\Delta^\lambda \in M_{d_\lambda}(\cO)$ is a diagonal matrix with
diagonal coefficients having positive real numbers as constant terms.
In particular, $\det(\Delta^\lambda) \in \cO^\times$.
\end{prop}

\begin{proof} We may assume without loss of generality that $F\subseteq\R$. 
Let $(\;,\;)$ be any symmetric bilinear form on $E^\lambda_\varepsilon$
which admits an orthonormal basis. We define  a new bilinear form $\langle
\;,\;\rangle$ 
by the formula
\[ \langle e,e'\rangle:=\sum_{w \in W} (T_w.e,T_w.e') \qquad 
\mbox{for any $e,e'\in E^\lambda_\varepsilon$}.\]
As in the proof of \cite[1.7]{Lusztig81b}, it is easily checked that 
$\langle T_s.e, e'\rangle=\langle e,T_s.e'\rangle$ for all $s \in S$ and, 
hence, $\langle T_w.e,e'\rangle=\langle e, T_{w^{-1}}.e'\rangle$ for all 
$w \in W$.  Arguing as in Step~1 of the proof of \cite[Prop.~4.3]{my02}, we 
see that the following holds:
\begin{equation*}
\mbox{for any $0 \neq e \in E^\lambda_\varepsilon$, we have 
$\varepsilon^{2g} \langle e,e \rangle \in b+\fp$}, \tag{$*$}
\end{equation*}
where $g\in \Gamma$ and $b \in F$ is such that $b>0$. (Recall that
$F \subseteq \R$.) Since we are working over a field of characteristic
$0$, there exists an orthogonal basis, $\{e_1,\ldots,e_{d_\lambda}\}$ say,
with respect to $\langle \;,\; \rangle$. Now ($*$) implies that, by 
multiplying the basis vectors $e_i$ by $\varepsilon^{-g_i}$ for suitable 
$g_i \in \Gamma$, we can assume that 
\[\langle e_i,e_i \rangle \in b_i+\fp \qquad \mbox{where $b_i \in F$, 
$b_i>0$}.\]
Let $\rho^\lambda$ be the matrix representation afforded by 
$E^\lambda_\varepsilon$ with respect to the basis $\{e_1,\ldots,
e_{d_\lambda}\}$ and let $\Delta^\lambda$ be the Gram matrix of
$\langle \;, \;\rangle$ with respect to that basis. Let $D^\lambda$
be the diagonal matrix with $b_1,\ldots,b_{d_\lambda}$ on the diagonal. 
Then we have
\[\Delta^\lambda \equiv D^\lambda \bmod \fp \quad \mbox{and}\quad 
\Delta^\lambda \,\rho^\lambda (T_{w^{-1}})=\rho^\lambda(T_w)^{\text{tr}} \, 
\Delta^\lambda \quad \mbox{for all $w \in W$}.\]
We can now argue as in the proof of \cite[Theorem~4.4]{my02} to show
that ${\rho}^\lambda$ is balanced. Indeed, let $\gamma \in \Gamma$
be minimal such that $\varepsilon^\gamma {\rho}^\lambda_{ij}(T_w) \in
\cO$ for all $w \in W$ and all $1\leq i,j \leq d_\lambda$. Let $\hat{c}_{w,
\lambda}^{ij} \in F$ be the constant term of $\varepsilon^\gamma
{\rho}^\lambda_{ij}(T_w)$. Choose $i,j \in \{1,\ldots,d_\lambda\}$ such
that $\hat{c}_{y,\lambda}^{ij}\neq 0$ for some $y \in W$. Now, we do not 
only have the orthogonality relations already mentioned above, but also 
the Schur relations in \cite[Cor.~7.2.2]{gepf}. Thus, we have 
\[ \varepsilon^{2\gamma} \bc_\lambda \equiv \sum_{w \in W}
\bigl(\varepsilon^\gamma{\rho}_{ij}^\lambda (T_w)\bigr)
\bigl(\varepsilon^\gamma {\rho}_{ji}^\lambda(T_{w^{-1}})\bigr)
\equiv \sum_{w \in W} \hat{c}_{w,\lambda}^{ij}\, \hat{c}_{w^{-1}\lambda}^{ji}
\bmod \fp.\]
Now we multiply the relation $\Delta^\lambda \,\rho^\lambda (T_{w^{-1}})=
\rho^\lambda(T_w)^{\text{tr}} \, \Delta^\lambda$ by 
$\varepsilon^{\ba_\lambda}$ and consider constant terms. Taking into account 
the relation $\Delta^\lambda \equiv D^\lambda \bmod \fp$, we obtain
\[b_j \hat{c}_{w^{-1},\lambda}^{ji}=\hat{c}_{w,\lambda}^{ij}b_i \qquad
\mbox{for all $w \in W$}.\]
This yields
\[ \sum_{w \in W} \hat{c}_{w,\lambda}^{ij}\, \hat{c}_{w^{-1}\lambda}^{ji}
=b_ib_j^{-1} \sum_{w\in W} \bigl(\hat{c}_{w,\lambda}^{ij} \bigr)^2,\]
which is a non-zero real number since $\hat{c}_{y,\lambda}^{ij} \neq 0$
for some $y \in W$. Thus, we conclude that $\varepsilon^{2\gamma}
\bc_\lambda$ lies in $\cO$ and has a non-zero constant term.
Comparing with the relation $\varepsilon^{2\ba_\lambda}\,\bc_\lambda\equiv
f_\lambda \bmod \fp$, we deduce that $\gamma=\ba_\lambda$ as required.
\end{proof}

\begin{rem} \label{bal0} In \cite[Prop.~4.3]{my02}, we assumed that $F=\R$.
This allowed us to go one step further in the above proof and take square 
roots of the numbers $b_i$. Consequently, by rescaling the basis vectors 
$e_i$, we can even assume that $\Delta^\lambda$ is diagonal with diagonal 
coefficients in $1+\fp$. The resulting balanced representations were 
called {\em orthogonal representations} in \cite{my02}. The corresponding 
leading matrix coefficients satisfy the following additional property
(see \cite[Theorem~4.4]{my02}):
\[ c_{w,\lambda}^{ij}=c_{w^{-1},\lambda}^{ji} \quad \mbox{for all
$w \in W$ and $1\leq i,j\leq d_\lambda$}.\]
\end{rem}

\subsection{The Kazhdan--Lusztig basis and Lusztig's $\ba$-function}
\label{sub23}

We now recall the basic facts about the Kazhdan--Lusztig basis of
$\bH$, following Lusztig \cite{Lusztig83}, \cite{Lusztig03}. Again, this
relies on the choice of a monomial $\leq$ on $\Gamma$. Now, there is a 
unique ring involution $A\rightarrow A$, $a \mapsto \bar{a}$, such that 
$\overline{\varepsilon^g}=\varepsilon^{-g}$ for all $g\in\Gamma$. We can 
extend this map to a ring involution $\bH \rightarrow \bH$, $h \mapsto 
\overline{h}$, such that
\[ \overline{\sum_{w \in W} a_w T_w}=\sum_{w \in W} \bar{a}_w
T_{w^{-1}}^{-1} \qquad (a_w \in A).\]
We define $\Gamma_{\geq 0}=\{g\in \Gamma\mid g\geq 0\}$ and denote by
$\Z[\Gamma_{\geq 0}]$ the set of all integral linear combinations of
terms $\varepsilon^g$ where $g\geq 0$. The notations $\Z[\Gamma_{>0}]$,
$\Z[\Gamma_{\leq 0}]$, $\Z[\Gamma_{<0}]$ have a similar meaning.
By Kazhdan--Lusztig \cite{KaLu} and Lusztig \cite{Lusztig83}, 
\cite{Lusztig03}, we have a ``new'' basis $\{C_w'\mid w \in W\}$ of
$\bH$ (depending on $\leq$), where $C_w'$ is characterised by the 
following two conditions:
\begin{itemize}
\item $\overline{C}_w'=C_w'$ and
\item $C_w'=T_w+\sum_{y \in W} p_{y,w} T_y$ where $p_{y,w}\in 
{\Z}[\Gamma_{<0}]$ for all $y \in W$.
\end{itemize}
Here we follow the original notation in \cite{KaLu}, \cite{Lusztig83}; the
element $C_w'$ is denoted by $c_w$ in \cite[Theorem~5.2]{Lusztig03}. As in 
\cite{Lusztig03}, it will be convenient to work with the following 
alternative version of the Kazhdan--Lusztig basis. We set 
$\bC_w=(C_w')^{\dagger}$ for all $w \in W$, where $\dagger\colon \bH 
\rightarrow \bH$ is the $A$-algebra automorphism defined by $T_s^\dagger=
-T_s^{-1}$ ($s \in S$); see \cite[3.5]{Lusztig03}. Note that $\overline{h}
=j(h)^\dagger=j(h^\dagger)$ for all $h \in \bH$ where $j \colon \bH 
\rightarrow\bH$ is the ring involution such that $j(a)=\bar{a}$ for 
$a \in A$ and $j(T_w)=(-1)^{l(w)}T_w$ for $w \in W$. Thus, we have 
\begin{itemize}
\item $\overline{\bC}_w=\bC_w$ and
\item $\bC_w=j(C_w')=(-1)^{l(w)}T_w +\sum_{y \in W} 
(-1)^{l(y)}\overline{p}_{y,w}
T_y$, where $\overline{p}_{y,w}\in {\Z}[\Gamma_{>0}]$.
\end{itemize}
Since the elements $\{\bC_w\mid w\in W\}$ form a basis of $\bH$, we can 
write
\[ \bC_x \bC_y=\sum_{z \in W} h_{x,y,z}\, \bC_z\qquad \mbox{for any
$x,y \in W$},\]
where $h_{x,y,z}=\overline{h}_{x,y,z} \in A$ for all $x,y,z\in W$.
Note that either $h_{x,y,z} \in \Z$ or $h_{x,y,z}$ involves terms from both
$\Gamma_{<0}$ and $\Gamma_{>0}$. For a fixed $z \in W$, we set
\[ \ba(z):= \min \{g\in\Gamma_{\geq 0}\mid \varepsilon^g\,h_{x,y,z} \in
\Z[\Gamma_{\geq 0}] \mbox{ for all $x,y\in W$}\}.\]
This is Lusztig's function $\ba \colon W \rightarrow \Gamma$; see 
\cite[Chap.~13]{Lusztig03}. Given $x,y,z\in W$, we have
$\varepsilon^{\ba(z)}\, h_{x,y,z} \in {\Z}[\Gamma_{\geq 0}]$. By 
\cite[13.9]{Lusztig03}, we have $\ba(z)=\ba(z^{-1})$. Then we define 
$\gamma_{x,y,z} \in \Z$ to be the constant term of $\varepsilon^{\ba(z)}\,
h_{x,y,z^{-1}} \in {\Z}[\Gamma_{\geq 0}]$, that is, we have
\[ \varepsilon^{\ba(z)}\, h_{x,y,z^{-1}} \equiv \gamma_{x,y,z} \bmod {\Z}
[\Gamma_{>0}].\]
These constants appear as the structure constants in Lusztig's ring
$\bJ$; see \cite[Chap.~18]{Lusztig03}.

We can now state the following result which relates the $\ba$-function 
and $\gamma_{x,y,z}$ to leading matrix coefficients. Here we assume that, 
for each $\lambda \in \Lambda$, we have chosen a balanced representation 
$\rho^\lambda$ afforded by $E^\lambda_\varepsilon$ as in Remark~\ref{bal0}. 
(We will see in Proposition~\ref{asymfin} that the same statement holds
for any choice of balanced representations.)

\begin{prop}[See \protect{\cite[Prop.~3.6 and Rem.~4.2]{klremarks}}] 
\label{klrem1} Assume that Lusztig's conjectures {\bf P1}--{\bf P15} in 
\cite[14.2]{Lusztig03} hold. Let $z \in W$. If $\lambda \in \Lambda$ and 
$i,j \in \{1,\ldots,d_\lambda\}$ are such that $c_{z,\lambda}^{ij} \neq 0$, 
then $\ba(z)=\ba_\lambda$. Furthermore, for all $x,y,z\in W$, we have 
\[ \gamma_{x,y,z}=\sum_{\lambda \in \Lambda} \sum_{1\leq i,j,k\leq 
d_\lambda} f_\lambda^{-1} \, c_{x,\lambda}^{ij}\, c_{y,\lambda}^{jk}\, 
c_{z,\lambda}^{ki}.\]
\end{prop}

In the next section, we will use the expression on the right hand side
of the above identity to construct a ring $\tilde{\bJ}$, {\em without} 
assuming  that {\bf P1}--{\bf P15} hold. Note also that not all of 
{\bf P1}--{\bf P15} are required for proving Proposition~\ref{klrem1}. 
For example, {\bf P15} is not needed; see \cite[Remark~3.9]{klremarks}.

\begin{rem} \label{note2} The conjectures {\bf P1}--{\bf P15} are known to 
hold, for example, in the equal parameter case. For crystallographic $W$, see
\cite[Chap.~16]{Lusztig03} and the references there. For $W$ of type 
$I_2(m)$, $H_3$ or $H_4$, see DuCloux \cite{Fokko}. Now let $(W,S)$ be of 
type $B_n$, $F_4$ or $I_2(m)$ ($m$ even). Let $L_0 \colon W \rightarrow 
\Gamma_0$ be the universal weight function as in Example~\ref{Mrem12}. Thus, 
$L_0$ depends on two values $a,b \in \Gamma$, which are attached to the 
generators in $S$:
\begin{center}
\makeatletter
\vbox{\begin{picture}(200,60)
\put( 10,40){$B_n$}
\put( 50,40){\@dot{5}}
\put( 48,47){$b$}
\put( 50,40){\line(1,0){20}}
\put( 59,43){$\scriptstyle{4}$}
\put( 70,40){\@dot{5}}
\put( 68,47){$a$}
\put( 70,40){\line(1,0){30}}
\put( 90,40){\@dot{5}}
\put( 88,47){$a$}
\put(110,40){\@dot{1}}
\put(120,40){\@dot{1}}
\put(130,40){\@dot{1}}
\put(140,40){\line(1,0){10}}
\put(150,40){\@dot{5}}
\put(147,47){$a$}

\put( 10, 12){$I_2(m)$}
\put( 10, 02){$\;\scriptstyle{m\, {\rm even}}$}
\put( 50, 07){\@dot{5}}
\put( 48, 14){$b$}
\put( 56, 10){$\scriptstyle{m}$}
\put( 50, 07){\line(1,0){20}}
\put( 70, 07){\@dot{5}}
\put( 68, 14){$a$}

\put(103, 12){$F_4$}
\put(130, 07){\@dot{5}}
\put(128, 14){$a$}
\put(130, 07){\line(1,0){20}}
\put(150, 07){\@dot{5}}
\put(148, 14){$a$}
\put(150, 07){\line(1,0){20}}
\put(158, 10){$\scriptstyle{4}$}
\put(170, 07){\@dot{5}}
\put(168, 14){$b$}
\put(170, 07){\line(1,0){20}}
\put(190, 07){\@dot{5}}
\put(188, 14){$b$}
\end{picture}
\makeatother}
\end{center}
Choose a pure lexicographic order on $\Gamma_0$, such that $b>ra>0$ for 
all $r\in\Z_{\geq 1}$. Then {\bf P1}--{\bf P15} are also known to hold; 
see \cite[Theorem~5.3]{klremarks} and the references there. In analogy
to Bonnaf\'e--Iancu \cite{BI}, this may be called the general
``asymptotic case''.
\end{rem}

\section{The ring $\tilde{\bJ}$} \label{secJtilde}

In this section, we show that the ``leading matrix coefficients'' 
associated to balanced representations as in Definition~\ref{bal}
can be used to construct a ring $\tilde{\bJ}$.  We keep the basic setting
of \S \ref{sub22}. Throughout this section we assume that, for each 
$\lambda \in \Lambda$, we are given a balanced representation 
$\rho^\lambda$ afforded by $E^\lambda_\varepsilon$, with corresponding
leading matrix coefficients $c_{w,\lambda}^{ij}$. 

\begin{defn} \label{tilde} For $w,x,y,z \in  W$, we set
\begin{align*}
\tilde{\gamma}_{x,y,z}&:=\sum_{\lambda \in \Lambda}\sum_{1\leq i,j,
k\leq d_\lambda} f_\lambda^{-1}\, c_{x,\lambda}^{ij}\, c_{y,\lambda}^{jk}
\, c_{z,\lambda}^{ki},\\
\tilde{n}_w &:=\sum_{\lambda\in \Lambda} \sum_{1\leq i\leq d_\lambda} 
f_\lambda^{-1} \,c_{w^{-1},\lambda}^{ii}.
\end{align*}
Let $\tilde{\bJ}$ be the $F$-vectorspace 
with basis $\{t_w\mid w\in W\}$. We define a bilinear product on 
$\tilde{\bJ}$ by
\[t_xt_y=\sum_{z\in W}\tilde{\gamma}_{x,y,z^{-1}}\,t_z\qquad (x,y\in W).\]
Let $\tilde{\cD}:=\{ w\in W \mid \tilde{n}_w\neq 0\}$. We define an element 
of $\tilde{\bJ}$ by $1_{\tilde{\bJ}}:=\sum_{w \in \tilde{\cD}} 
\tilde{n}_w\, t_w$. 

Note that the above definitions appear to depend on the choice of
$\rho^\lambda$ but at the end of this section, we will see that this
is not the case. 
\end{defn}

\begin{rem} \label{bal1} Since $\bH$ is symmetric, we have the following 
Schur relations (see \cite[Cor.~7.2.2]{gepf}):
\begin{equation*}
\sum_{y \in W} \rho_{ij}^\lambda(T_w)\,\rho^\mu_{kl}(T_{w^{-1}})=
\delta_{il}\delta_{jk} \delta_{\lambda\mu} \bc_\lambda,
\end{equation*}
where $\lambda,\mu \in \Lambda$, $1\leq i,j\leq d_\lambda$ and $1\leq k,l
\leq d_\mu$. Multiplying by $\varepsilon^{\ba_\lambda+\ba_\mu}$ and taking 
constant terms on both sides, we obtain orthogonality relations for the 
leading matrix coefficients:
\begin{equation*}
\sum_{w \in W} c_{w,\lambda}^{ij}\, c_{w^{-1},\mu}^{kl}=
\delta_{il}\delta_{jk} \delta_{\lambda\mu} f_\lambda.\tag{$*$}
\end{equation*}
These relations can be ``inverted'' and so we also have:
\begin{equation*}
\sum_{\lambda\in \Lambda} \sum_{1\leq i,j\leq d_\lambda} f_\lambda^{-1}
c_{x,\lambda}^{ij} \, c_{y^{-1},\lambda}^{ji}=\delta_{xy} \quad \mbox{for
all $x,y \in W$}.\tag{$*^\prime$}
\end{equation*}
\end{rem}

\begin{lem} \label{asym0} We have the following relations:
\begin{alignat*}{2}
\tilde{\gamma}_{x,y,z}&=\tilde{\gamma}_{y,z,x} &&\qquad \mbox{for all
$x,y,z\in W$},\tag{a}\\
\sum_{w\in W} \tilde{\gamma}_{x^{-1},y,w}\, \tilde{n}_w
&=\delta_{xy} &&\qquad \mbox{for all $x,y \in W$}. \tag{b}
\end{alignat*}
\end{lem}

\begin{proof} (a) Just note that the defining formula for $\tilde{\gamma}_{x,
y,z}$ is symmetrical under cyclic permutations of $x,y,z$. 

(b) Using the defining formulas for $\tilde{\gamma}_{x,y,z}$ and
$\tilde{n}_w$, the left hand side evaluates to
\begin{align*}
\Bigl(\sum_{\lambda \in \Lambda} &\sum_{1\leq i,j, k\leq 
d_\lambda} f_\lambda^{-1}\, c_{x^{-1},\lambda}^{ij}\, 
c_{y,\lambda}^{jk} \, c_{w,\lambda}^{ki}\Bigr) \Bigl(\sum_{w\in W}
\sum_{\mu \in \Lambda} \sum_{1\leq p\leq d_\mu} f_\mu^{-1} \, c_{w^{-1},
\mu}^{pp}\Bigr) \\ &= \sum_{\lambda,\mu \in \Lambda} \sum_{1\leq i,j,k\leq 
d_\lambda} \sum_{1\leq p\leq d_\mu} f_\lambda^{-1}\, f_\mu^{-1}\, 
c_{x^{-1},\lambda}^{ij}\, c_{y,\lambda}^{jk} \Bigl( \sum_{w \in W} 
c_{w,\lambda}^{ki}\, c_{w^{-1},\mu}^{pp}\Bigr).
\end{align*}
By the relations in Remark~\ref{bal1}($*$), the parenthesized
sum evaluates to $\delta_{kp} \delta_{ip}\delta_{\lambda\mu}f_\lambda$. 
Inserting this into the above expression yields
$\sum_{\lambda\in \Lambda} \sum_{1\leq i,j\leq d_\lambda} f_\lambda^{-1}\, 
c_{x^{-1},\lambda}^{ij}\, c_{y,\lambda}^{ji}=\delta_{xy}$,
where the last equality holds by Remark~\ref{bal1}($*^\prime$).
\end{proof}

\begin{prop} \label{asym1} $\tilde{\bJ}$ is an  associative algebra 
with identity element $1_{\tilde{\bJ}}$. 
\end{prop}

\begin{proof} 
Let $x,y,z\in W$. We must check that $(t_xt_y)t_z=t_x(t_yt_z)$, which
is equivalent to 
\[ \sum_{u \in W} \tilde{\gamma}_{x,y,u^{-1}}\, \tilde{\gamma}_{u,
z,w^{-1}}=\sum_{u \in W} \tilde{\gamma}_{x,u, w^{-1}}\, 
\tilde{\gamma}_{y,z,u^{-1}}.\]
Using the defining formula, the left hand side evaluates to
\begin{align*}
\sum_{u \in W} &\Bigl(\sum_{\lambda \in \Lambda}\sum_{1\leq i,j,k\leq 
d_\lambda} f_\lambda^{-1}\, c_{x,\lambda}^{ij}\, c_{y,\lambda}^{jk}
\, c_{u^{-1},\lambda}^{ki}\Bigr)\Bigl(\sum_{\mu \in \Lambda}\sum_{1\leq 
p,q,r \leq d_\mu} f_\mu^{-1}\, c_{u,\lambda}^{pq}\, c_{z,\lambda}^{qr}
\, c_{w^{-1},\lambda}^{rp}\Bigr)\\
&=\sum_{\lambda,\mu\in \Lambda} \sum_{1\leq i,j,k\leq d_\lambda} 
\sum_{1\leq p,q,r\leq d_\mu} f_\lambda^{-1}\,f_\mu^{-1}\,c_{x,\lambda}^{ij}\, 
c_{y,\lambda}^{jk}\, c_{z,\lambda}^{qr} \, c_{w^{-1},\lambda}^{rp} 
\Bigl(\sum_{u \in W} c_{u^{-1},\lambda}^{ki}\, c_{u,\lambda}^{pq}\Bigr).
\end{align*}
By the relations in Remark~\ref{bal1}($*$), the parenthesized sum
evaluates to $\delta_{kq}\delta_{pi}\delta_{\lambda\mu}f_\lambda$. 
Hence, the above expression equals
\[\sum_{\lambda\in \Lambda} \sum_{1\leq i,j,k,r\leq d_\lambda}
f_\lambda^{-1}\,c_{x,\lambda}^{ij}\, c_{y,\lambda}^{jk}\,
 c_{z,\lambda}^{kr} \, c_{w^{-1},\lambda}^{ri}.\]
By a similar computation, the right hand side evaluates to
\begin{align*}
\sum_{u \in W} &\Bigl(\sum_{\lambda \in \Lambda}\sum_{1\leq i,j,k\leq 
d_\lambda} f_\lambda^{-1}\, c_{x,\lambda}^{ij}\, c_{u,\lambda}^{jk}
\, c_{w^{-1},\lambda}^{ki}\Bigr)\Bigl(\sum_{\mu \in \Lambda}\sum_{1\leq
p,q,r \leq d_\mu} f_\mu^{-1}\, c_{y,\lambda}^{pq}\, c_{z,\lambda}^{qr}
\, c_{u^{-1},\lambda}^{rp}\Bigr)\\ &=\sum_{\lambda,\mu \in \Lambda}
\sum_{1\leq i,j,k\leq d_\lambda} \sum_{1\leq p,q,r\leq d_\mu} 
f_\lambda^{-1}\, f_\mu^{-1}\, c_{x,\lambda}^{ij}\, c_{w^{-1},\lambda}^{ki}\, 
c_{y,\lambda}^{pq}\, c_{z,\lambda}^{qr} \Bigl(\sum_{u \in W} 
c_{u,\lambda}^{jk}\, c_{u^{-1},\lambda}^{rp}\Bigr)\\
&=\sum_{\lambda \in \Lambda} \sum_{1\leq i,j,k,q\leq d_\lambda}
f_\lambda^{-1}\, c_{x,\lambda}^{ij}\, c_{y,\lambda}^{jq}\,
c_{z,\lambda}^{qk}\,c_{w^{-1},\lambda}^{ki}.
\end{align*}
We see that both sides are equal, hence $\tilde{\bJ}$ is associative. To 
show that $1_{\tilde{\bJ}}$ is the identity element of $\tilde{\bJ}$, we
let $x \in W$ and note that
\begin{align*}
t_x1_{\tilde{\bJ}}&=\sum_{w\in W} \tilde{n}_w\, t_xt_w=\sum_{y \in W}
\Bigl(\sum_{w\in W} \tilde{n}_w\, \tilde{\gamma}_{x,w,y^{-1}}\Bigr)t_y\\
&=\sum_{y \in W} \Bigl(\sum_{w\in W} \tilde{n}_w\, \tilde{\gamma}_{y^{-1},
x,w}\Bigr)t_y=t_x \qquad \mbox{by Lemma~\ref{asym0}(a) and (b)}.
\end{align*}
A similar argument shows that $1_{\tilde{\bJ}}t_x=t_x$.  Thus, 
$1_{\tilde{\bJ}}$ is the identity element of $\tilde{\bJ}$.
\end{proof}

\begin{prop} \label{asym2} The linear map $\bar{\tau} \colon 
\tilde{\bJ} \rightarrow F$ defined by $\bar{\tau}(t_w)= 
\tilde{n}_{w^{-1}}$ is a symmetrizing trace such that 
$\bar{\tau}(t_xt_{y^{-1}})=\delta_{xy}$ for all $x,y\in W$.
\end{prop}

\begin{proof} Let $x,y \in W$. Then, using Lemma~\ref{asym0}(b), we obtain
\[ \bar{\tau}(t_{x^{-1}}t_y)=\sum_{w \in W} \tilde{\gamma}_{x^{-1},y,
w^{-1}}\, \bar{\tau}(t_w)=\sum_{w \in W} \gamma_{x^{-1},y,w^{-1}}\, 
\tilde{n}_{w^{-1}}=\delta_{xy}.\]
This implies that $\bar{\tau}(t_xt_y)=\bar{\tau}(t_yt_x)$ for all $x,y
\in W$, hence $\bar{\tau}$ is a trace function. We also see that
$\{t_w \mid w \in W\}$ and $\{t_{w^{-1}}\mid w \in W\}$ form a pair of dual 
bases, hence $\bar{\tau}$ is non-degenerate. Thus, $\tilde{\bJ}$
is a symmetric algebra with trace form $\bar{\tau}$.
\end{proof}

\begin{prop} \label{asym3} For $\lambda \in \Lambda$, define a linear map 
\[\bar{\rho}^\lambda \colon \tilde{\bJ}\rightarrow M_{d_\lambda}(F),\qquad 
t_w \mapsto \bigl(c_{w,\lambda}^{ij}\bigr)_{1\leq i,j \leq d_\lambda}.\]
Then $\bar{\rho}^\lambda$ is an absolutely irreducible representation of 
$\tilde{\bJ}$, and all irreducible representations of $\tilde{\bJ}$ (up 
to equivalence) arise in this way. In particular, $\tilde{\bJ}$ is a 
split semisimple algebra. (Recall that $F$ is any field containing
$\Z_W$.)
\end{prop}

\begin{proof}
We must show that $\bar{\rho}^\lambda(t_xt_y)=\bar{\rho}^\lambda(t_x)
\bar{\rho}^\lambda(t_y)$ for all $x,y\in W$. Now, by the definition of
$\tilde{\gamma}_{x,y,z}$, we have
\[\bar{\rho}_{ij}^\lambda(t_xt_y)=\sum_{z \in W} \tilde{\gamma}_{x,y,
z^{-1}}\, c_{z,\lambda}^{ij}=\sum_{z \in W} \Bigl(\sum_{\mu\in \Lambda}
\sum_{1\leq p,q,r\leq d_\mu} f_\mu^{-1}\,c_{x,\mu}^{pq}\,c_{y,\mu}^{qr}\,
c_{z^{-1},\mu}^{rp}\Bigr) c_{z,\lambda}^{ij}.\]
Using the Schur relations in Remark~\ref{bal1}($*$), the right hand side
evaluates to 
\[\sum_{\mu\in \Lambda}\sum_{1\leq p,q,r\leq d_\mu} f_\mu^{-1}
\,c_{x,\mu}^{pq}\,c_{y,\mu}^{qr}\,\delta_{rj}\delta_{pi}\delta_{\lambda\mu}
f_\lambda=\sum_{1\leq q\leq d_\lambda} \,c_{x,\mu}^{iq}\,c_{y,\mu}^{qj}=
\bigl(\bar{\rho}^\lambda(t_x)\bar{\rho}^\lambda(t_y)\bigr)_{ij},\]
as required. To show that $\bar{\rho}^\lambda$ is absolutely irreducible, 
we argue as follows. By Proposition~\ref{asym2}, we have a symmetrizing 
trace where $\{t_w \mid w \in W\}$ and $\{t_{w^{-1}}\mid w \in W\}$ form a 
pair of dual bases. Consequently, the relations in Remark~\ref{bal1}($*$) 
can be interpreted as orthogonality relations for the coefficients of the 
representations $\bar{\rho}^\lambda$. Thus, we have:
\[ \sum_{w \in W} \bar{\rho}^\lambda_{ij}(t_w)\, \bar{\rho}^\lambda_{kl}
(t_{w^{-1}})=\delta_{il}\delta_{jk}f_\lambda \quad \mbox{for all
$1\leq i,j,k,l\leq d_\lambda$}.\]
By \cite[Remark~7.2.3]{gepf}, the validity of these relations implies
that $\bar{\rho}^\lambda$ is absolutely irreducible.
Finally, if $\lambda\neq \mu$ in $\Lambda$, then we also have the 
relations:
\[ \sum_{w \in W} \bar{\rho}^\lambda_{ij}(t_w)\, \bar{\rho}^\mu_{kl}
(t_{w^{-1}})=0.\]
In particular, this implies that $\bar{\rho}^\lambda$ and 
$\bar{\rho}^\mu$ are not equivalent. 

Since $\dim \tilde{\bJ}=|W|=\sum_{\lambda \in \Lambda} d_\lambda^2$, we can 
now conclude that $\tilde{\bJ}$ is split semisimple, and that 
$\{\bar{\rho}^\lambda \mid \lambda \in \Lambda\}$ are the irreducible 
representations of $\tilde{\bJ}$ (up to equivalence).
\end{proof}

We can now settle the question to what extent the ring $\tilde{\bJ}$
depends on the choice of the balanced representations $\rho^\lambda$.

\begin{lem} \label{bal3} Assume that $\rho^\lambda$ and $\sigma^\lambda$
are balanced and equivalent over $K$. Then there exists a matrix
$U^\lambda \in M_{d_\lambda}(\cO)$ such that 
\[ \det(U^\lambda)\in \cO^\times \quad \mbox{and}\quad 
U^\lambda\,\rho^\lambda(T_w)=\sigma^\lambda(T_w)\, U^\lambda
\quad \mbox{for all $w \in W$}.\]
Denote the leading matrix coefficients with respect to $\sigma^\lambda$ by 
$d_{w,\lambda}^{ij}$. Then, for a given element $w \in W$, we have 
\[ c_{w,\lambda}^{ij}\neq 0 \mbox{ for some  $i,j$} \quad
\Leftrightarrow \quad d_{w,\lambda}^{kl} \neq 0 \mbox{ for some $k,l$}.\]
\end{lem}

\begin{proof} Since $\rho^\lambda$ and $\sigma^\lambda$ are equivalent 
over $K$, there exists an invertible matrix $U^\lambda \in M_{d_\lambda}(K)$
such that $U^\lambda \rho^\lambda(T_w)=\sigma^\lambda(T_w)U^\lambda$
for all $\lambda \in \Lambda$. Multiplying $U^\lambda$ by a suitable scalar,
we may assume that all coefficients of $U^\lambda$ lie in $\cO$ and that
at least one coefficient does not lie in $\fp$. 

We show that $\det(U^\lambda) \in \cO^\times$. For this purpose, let 
$\bar{U}^\lambda$ be the matrix whose $(i,j)$-coefficient is the constant 
term of the $(i,j)$-coefficient of $U^\lambda$. Multiplying the relation 
$U^\lambda\rho^\lambda (T_w)=\sigma^\lambda (T_w) U^\lambda$ by 
$\varepsilon^{\ba_\lambda}$ and taking constant terms, we see that 
$\bar{U}^\lambda\in M_{d_\lambda}(F)$ is a non-zero matrix such that 
\[ \bar{U}^\lambda\,\bar{\rho}^\lambda(t_w)=\bar{\sigma}^\lambda(t_w)\, 
\bar{U}^\lambda \quad \mbox{for all $w \in W$},\]
where $\bar{\sigma}^\lambda(t_w):=(d_{w,\lambda}^{ij})_{1\leq
i,j\leq d_\lambda}$. (Note that, at this stage, we do not know yet if 
$\bar{\sigma}^\lambda$ is a representation of $\tilde{\bJ}$ but in any 
case, this is irrelevant for the argument to follow.) Now let $v \in 
F^{d_\lambda}$ be such that $\bar{U}^\lambda v=0$. Then we also have 
\[ \bar{U}^\lambda\,\bigl(\bar{\rho}^\lambda(t_w)v\bigr)=
\bar{\sigma}^\lambda(t_w)\, \bar{U}^\lambda v=0,\]
and so the nullspace of $\bar{U}^\lambda$ is a $\bar{\rho}^\lambda$-invariant
subspace of $U^{d_\lambda}$. Since $\bar{\rho}^\lambda$ is irreducible
and $\bar{U}^\lambda \neq 0$, we conclude that the nullspace is $0$
and, hence, $\bar{U}^\lambda$ is invertible, as claimed.

The assertion about the leading matrix coefficients is now clear.
\end{proof}

\begin{prop} \label{asymfin} The ring $\tilde{\bJ}$ does not depend
on the choice of the balanced representations $\{\rho^\lambda\mid
\lambda \in \Lambda\}$.
\end{prop}

\begin{proof} Using the notation in Proposition~\ref{asym3}, the
defining formulas in Definition~\ref{tilde} read:
\begin{align*}
\tilde{\gamma}_{x,y,z}&=\sum_{\lambda \in \Lambda}
\sum_{1\leq i,j,k\leq d_\lambda} f_\lambda^{-1} \,c_{x,\lambda}^{ij}
\, c_{y,\lambda}^{jk}\, c_{z,\lambda}^{ki}=\sum_{\lambda \in
\Lambda} f_\lambda^{-1}\,\mbox{trace}\bigl(\bar{\rho}^\lambda(t_x)\,
\bar{\rho}^\lambda(t_y) \, \bar{\rho}^\lambda(t_z)\bigr),\\
\tilde{n}_w &=\sum_{\lambda\in \Lambda} \sum_{1\leq i\leq d_\lambda}
f_\lambda^{-1} \,c_{w^{-1},\lambda}^{ii}=\sum_{\lambda\in \Lambda}
f_\lambda^{-1}\, \mbox{trace}\bigl(\bar{\rho}^\lambda(t_{w^{-1}})\bigr).
\end{align*}
Now assume that $\{\sigma^\lambda \mid \lambda \in \Lambda\}$ also is
a collection of balanced representations where $\rho^\lambda$ and
$\sigma^\lambda$ are equivalent over $K$. Let $d_{w,\lambda}^{ij}$ be
the leading matrix coefficients defined with respect to $\sigma^\lambda$,
and set $\bar{\sigma}^\lambda(t_w):=\bigl(d_{w,\lambda}^{ij}\bigr)_{1\leq 
i,j\leq d_\lambda}$. By the proof of Lemma~\ref{bal3}, there exist invertible 
matrices $\bar{U}^\lambda \in M_{d_\lambda}(F)$ such that 
\[ \bar{\sigma}^\lambda(T_w)=(\bar{U}^\lambda)^{-1}\,\bar{\rho}^\lambda(t_w)
\, \bar{U}^\lambda \quad \mbox{for all $w \in W$}.\]
Hence the above expressions immediately show that $\tilde{\gamma}_{x,y,z}$
and $\tilde{n}_w$ are independent of whether we use $\rho^\lambda$ or
$\sigma^\lambda$ to define them.
\end{proof}

\begin{prop} \label{asym4} The linear map $\tilde{\bJ}\rightarrow 
\tilde{\bJ}$ defined by $t_w \mapsto t_{w^{-1}}$ is an anti-involution, 
that is, we have $\tilde{\gamma}_{x,y,z}=\tilde{\gamma}_{y^{-1},
x^{-1},z^{-1}}$ for all $x,y,z\in W$.
\end{prop}

\begin{proof} By Proposition~\ref{asymfin}, we may assume that $F=\R$ and 
that our balanced representations $\rho^\lambda$ are chosen such that they 
are orthogonal, as in Remark~\ref{bal0}. Then the corresponding leading 
matrix coefficients have the additional property $c_{w,\lambda}^{ij}=
c_{w^{-1}, \lambda}^{ji}$. The defining formula then immediately shows that 
$\tilde{\gamma}_{x,y,z}=\tilde{\gamma}_{y^{-1}, x^{-1},z^{-1}}$ for all 
$x,y,z\in W$. 
\end{proof}

\begin{rem} \label{preordL} Let $\lambda \in \Lambda$ and $w \in W$. As in 
\cite[Def.~3.1]{klremarks}, we write $E^\lambda \leftrightsquigarrow_L w$
if $c_{w,\lambda}^{ij}\neq 0$ for some $i,j\in \{1,\ldots,d_\lambda\}$.
By Lemma~\ref{bal3}, this relation does not depend on the choice of
the balanced representations $\rho^\lambda$. In particular, choosing 
$\rho^\lambda$ as in  Remark~\ref{bal0}, we see that 
\begin{equation*}
E^\lambda\leftrightsquigarrow_L w\qquad\Leftrightarrow\qquad 
E^\lambda\leftrightsquigarrow_L w^{-1}.  \tag{a}
\end{equation*}
Now define a graph as follows: The vertices are in bijection with the
elements of $W$; two vertices corresponding to elements $x\neq y$ in $W$ 
are joined by an edge if there exists some $\lambda \in \Lambda$ such 
that $E^\lambda \leftrightsquigarrow_{L} x$ and $E^\lambda 
\leftrightsquigarrow_{L} y$. Considering the connected components of this 
graph, we obtain a partition of $W$; the pieces in this partition will 
be called the {\em $L$-blocks} of $W$. By \cite[Lemma~3.2]{klremarks}, we
have that
\begin{equation*}
\mbox{each $L$-block is contained in a two-sided cell of $W$.}\tag{b}
\end{equation*}
(See \cite[Chap.~8]{Lusztig03} for the definition of two-sided cells; if 
{\bf P1}--{\bf P15} hold, then one can show that the $L$-blocks are 
precisely the two-sided cells of $W$.)

For an $L$-block $\fF$ of $W$, we define $\tilde{\bJ}_{\fF}=\langle t_w 
\mid w \in \fF\rangle_F \subseteq \tilde{\bJ}$. Then one easily checks
that $\tilde{\bJ}_{\fF}$ is a two-sided ideal of $\tilde{\bJ}$. (Indeed,
let $x\in W$, $w \in \fF$; we must show that $t_xt_w$ and $t_wt_x$ lie in
$\tilde{\bJ}_{\fF}$. Now, $t_xt_w=\sum_{y \in W} \tilde{\gamma}_{x,w,y^{-1}}
t_y$. Assume that $\tilde{\gamma}_{x,w,y^{-1}} \neq 0$. Then, by the 
defining formula, there exists some $\lambda \in \Lambda$ such that 
$E^\lambda \leftrightsquigarrow_L x$, $E^\lambda\leftrightsquigarrow_L w$ 
and $E^\lambda \leftrightsquigarrow_{L} y^{-1}$. By (a), we also have 
$E^\lambda \leftrightsquigarrow_L y$. It follows that $x,y,y^{-1} \in 
\fF$. Thus, $t_xt_w \in \tilde{\bJ}_{\fF}$. The argument for $t_wt_x$ is 
similar.) We obtain a decomposition as a direct sum of two-sided ideals
\begin{equation*}
\tilde{\bJ}={\displaystyle \bigoplus}_{\fF} \tilde{\bJ}_{\fF} \qquad
\mbox{(sum over all $L$-blocks $\fF$ of $W$)}.\tag{c}
\end{equation*}
Now, given $\lambda \in \Lambda$, there will be a unique $L$-block 
$\fF$ such that $\bar{\rho}^\lambda(t_w) \neq 0$ for some $w \in \fF$. We
denote this $L$-block by $\fF_\lambda$. 
\end{rem}

%

\section{Properties of balanced representations} \label{seclead}

The purpose of this section is to study in more detail balanced
representations as in Definition~\ref{bal}. In particular, we wish to
develop some methods for verifying if a given matrix representation is
balanced or not. The criterion in Proposition~\ref{bal2} will prove
very useful in dealing with a number of examples. Proposition~\ref{lem1}
exhibits some basic integrality properties.

We keep the general assumptions of the previous section. In particular,
$\{\rho^\lambda\mid \lambda \in \Lambda\}$ is a fixed choice of balanced 
representations of $\bH_K$.

\begin{lem} \label{bal2b} Let $\{\delta^\lambda \mid \lambda \in \Lambda\}$
be a complete set of representatives for the equivalences classes of 
irreducible representations of $\tilde{\bJ}$. Then $\rho^\lambda$ can be 
chosen such that $\bar{\rho}^\lambda (t_w)=\delta^\lambda(t_w)$ for all 
$w \in W$. 
\end{lem}

\begin{proof} First of all, we can assume without loss of generality that
$\bar{\rho}^\lambda$ is equivalent to $\delta^\lambda$ for each 
$\lambda \in \Lambda$. Let $G^\lambda \in M_{d_\lambda}(F)$ be  an
invertible matrix such that $\delta^\lambda(t_w)=(G^\lambda)^{-1}
\bar{\rho}^\lambda(t_w)G^\lambda$ for all $w \in W$. Now set
$\hat{\rho}^\lambda(T_w):=(G^\lambda)^{-1}\,\rho^\lambda(T_w)\,G^\lambda$
for $w \in W$. Then $\hat{\rho}^\lambda$ is an irreducible representation
of $\bH_K$ equivalent to $\rho^\lambda$. Moreover, since the 
transforming matrix $G^\lambda$ has all its coefficients in $F$, it 
is clear that $\hat{\rho}^\lambda$ is also balanced and that the leading
matrix coefficients associated with $\hat{\rho}^\lambda(T_w)$ are given 
by $\delta^\lambda(t_w)$. It remains to use Proposition~\ref{asymfin}.
\end{proof}

\begin{exmp} \label{leadp115} Assume that we are in the equal parameter
case or, more generally, that Lusztig's {\bf P1}--{\bf P15} are known
to hold; see Remark~\ref{note2}. Then, by Proposition~\ref{klrem1}, we have
\[ \tilde{\gamma}_{x,y,z}=\gamma_{x,y,z} \in \Z \qquad \mbox{for all
$x,y,z\in W$}.\]
Assume further that $R:=\Z_W$ is a principal ideal domain. Then, by
a general argument (see, e.g., \cite[7.3.7]{gepf}), every irreducible
representation of $\tilde{\bJ}$ can be realised over $R$. Hence, by
Lemma~\ref{bal2b}, the balanced representations of $\bH_K$ can be chosen 
such that 
\[ \bar{\rho}^\lambda(t_w) \in M_{d_\lambda}(\Z_W)\qquad \mbox{for all 
$\lambda \in \Lambda$ and $w \in W$}.\]
This applies to all finite Weyl groups in the equal parameter case, where
$\Z_W=W$. It also applies to $(W,S)$ of type $H_3$ or $H_4$, where $\Z_W=
{\Z}[\frac{1}{2}(-1+\sqrt{5})]$; note that $\Z_W$ is a principal ideal domain. 
By \cite[Theorem~5.2]{klremarks}, it also applies to $(W,S)$ of type
$F_4$ (where $\Z_W=\Z$), any weight function and any monomial order
on $\Gamma$.
\end{exmp}

\begin{prop} \label{bal2} Assume that $F \subseteq \R$ (which we can do
without loss of generality). Let $\lambda \in \Lambda$ and $\sigma^\lambda 
\colon \bH_K \rightarrow M_{d_\lambda}(K)$ be any matrix representation 
afforded by $E^\lambda_\varepsilon$. Then $\sigma^\lambda$ is balanced if and 
only if there exists a symmetric matrix $\Omega^\lambda \in M_{d_\lambda}
(\cO)$ such that 
\[ \det(\Omega^\lambda)\in \cO^\times \quad \mbox{and}\quad \Omega^\lambda 
\, \sigma^\lambda(T_{w^{-1}})=\sigma^\lambda(T_w)^{\operatorname{tr}} \, 
\Omega^\lambda \quad \mbox{for all $w \in W$}.\]
\end{prop}

\begin{proof} Assume first that $\sigma^\lambda$ is balanced. Now 
$\sigma^\lambda$ is obtained by choosing some basis of 
$E_\varepsilon^\lambda$. Let $\Omega^\lambda$ be the Gram matrix of 
$\langle \;,\;\rangle$ with respect to that basis, where $\langle \;,\;
\rangle$ is a bilinear form on $E^\lambda_\varepsilon$ as constructed 
in the proof of Proposition~\ref{balex}. Multiplying $\Omega^\lambda$ by 
a suitable scalar, we may assume without loss of generality that all 
coefficients of $\Omega^\lambda$ lie in $\cO$ and that some coefficient 
of $\Omega^\lambda$ does not lie in $\fp$. Then $\Omega^\lambda\in 
M_{d_\lambda} (\cO)$ is a symmetric matrix such that 
\[\Omega^\lambda\neq 0\quad\mbox{and}\quad \Omega^\lambda \, \sigma^\lambda
(T_{w^{-1}})=\sigma^\lambda(T_w)^{\text{tr}} \, \Omega^\lambda \quad 
\mbox{for all $w \in W$}.\]
Let $\bar{\Omega}^\lambda$ be the matrix whose $(i,j)$-coefficient is the
constant term of the $(i,j)$-coefficient of $\Omega^\lambda$. Now, 
multiplying the relation $\Omega^\lambda\,\sigma^\lambda (T_{w^{-1}})=
\sigma^\lambda (T_w)^{\text{tr}} \, \Omega^\lambda$ by 
$\varepsilon^{\ba_\lambda}$ and taking constant terms, we see that 
$\bar{\Omega}^\lambda$ is a non-zero symmetric matrix such that 
\begin{equation*}
\bar{\Omega}^\lambda \,\bar{\sigma}^\lambda(t_{w^{-1}})=\bar{\sigma}^\lambda
(t_w)^{\text{tr}} \, \bar{\Omega}^\lambda \qquad \mbox{for all $w \in W$}.
\tag{$*$}
\end{equation*}
Thus, $\bar{\Omega}^\lambda$ defines a $\tilde{\bJ}$-invariant symmetric
bilinear form on a representation space affording $\bar{\sigma}^\lambda$.
The invariance implies that the radical of the form is a 
$\tilde{\bJ}$-submodule. Hence, since $\bar{\sigma}^\lambda$ is an 
irreducible representation, we conclude that the radical must be zero
and so $\det(\bar{\Omega}^\lambda)\neq 0$. 

Conversely, assume that a matrix $\Omega^\lambda$ with the above properties 
exists. Let $\bar{\Omega}^\lambda$ be the matrix whose $(i,j)$-coefficient
is the constant term of the $(i,j)$-coefficient of $\Omega^\lambda$. Then
$\bar{\Omega}^\lambda\in M_{d_\lambda}(F)$ is a symmetric matrix such
that $\det(\bar{\Omega}^\lambda)\neq 0$ (since $\det(\Omega^\lambda)\in 
\cO^\times$). Thus, $\bar{\Omega}^\lambda$ defines a non-degenerate
symmetric bilinear form. Now, since we are working over a field of 
characteristic $0$, there will be an orthogonal basis with respect to that
form. So we can find invertible matrices $Q^\lambda, D^\lambda \in 
M_{d_\lambda}(F)$ such that 
$\bar{\Omega}^\lambda=(Q^\lambda)^{\text{tr}}\, D^\lambda \, Q^\lambda$ and 
$D^\lambda$ is diagonal. Now let $P^\lambda:=(Q^\lambda)^{-1}$ and define
\begin{align*}
\hat{\sigma}^\lambda(T_w)&:=(P^\lambda)^{-1}\,\sigma^\lambda(T_w)\,P^\lambda
\quad \mbox{for all $w \in W$},\\
\hat{\Omega}^\lambda &:=(P^\lambda)^{\text{tr}}\,\Omega^\lambda\,
P^\lambda.
\end{align*}
Thus, $\hat{\sigma}^\lambda$ is an irreducible 
representation of $\bH_K$ equivalent to $\sigma^\lambda$; furthermore, we have 
\[ \hat{\Omega}^\lambda \, \hat{\sigma}^\lambda(T_{w^{-1}})=
\hat{\sigma}^\lambda(T_w)^{\text{tr}} \, \hat{\Omega}^\lambda \quad 
\mbox{for all $w \in W$}.\]
Since the transforming matrix $P^\lambda$ has all its coefficients in $F$, 
it is clear that $\hat{\Omega}^\lambda \in M_{d_\lambda}(\cO)$ and
$\det(\hat{\Omega}^\lambda)\in \cO^\times$; furthermore, $\sigma^\lambda$ is 
balanced if and only if $\hat{\sigma}^\lambda$ is balanced. 

Thus, it remains to show that $\hat{\sigma}^\lambda$ is balanced. Now, the 
point about the above transformation is that we have $\hat{\Omega}^\lambda 
\equiv D^\lambda \bmod \fp$. We can now argue as in the proof of 
Proposition~\ref{balex} to show that $\hat{\sigma}^\lambda$ is balanced. 
\end{proof}


\begin{rem} \label{bal2a} Note that, in order to verify that a matrix
$\Omega^\lambda$ satisfies 
\[\Omega^\lambda \, \sigma^\lambda(T_{w^{-1}})=\sigma^\lambda(T_w)^{\text{tr}}
\, \Omega^\lambda \quad \mbox{for all $w \in W$},\]
it is sufficient to verify that $\Omega^\lambda \, \sigma^\lambda(T_s)=
\sigma^\lambda(T_s)^{\text{tr}} \, \Omega^\lambda$ for all $s \in S$.
This remark, although almost trivial, is nevertheless useful in
dealing with concrete examples.
\end{rem}

\begin{exmp} \label{leadingI2} Let $3 \leq m<\infty$ and $(W,S)$ be of type 
$I_2(m)$, with generators $s_1,s_2$ such that $(s_1s_2)^m=1$. We have $\Z_W=
{\Z}[\zeta+ \zeta^{-1}]$, where $\zeta \in \C$ is a root of unity of 
order $m$. In the sequel, we assume without loss of generality that
$L(s_1)\geq L(s_2)>0$. The irreducible representations of $\bH_K$ are
determined in \cite[\S 8.3]{gepf}. These representations have dimension
$1$ or $2$. Notice that $1$-dimensional representations are automatically 
balanced. By \cite[Theorem~8.3.1]{gepf}, the $2$-dimensional representations 
can be realized as 
\[ \rho_j \colon T_{s_1} \mapsto \left(\begin{array}{cc} -v_{s_1}^{-1} & 
0 \\ \mu_j & v_{s_2} \end{array}\right), \qquad T_{s_2} \mapsto 
\left(\begin{array}{cc} v_{s_2} & 1\\ 0 & -v_{s_2}^{-1}\end{array}\right),\]
where $\mu_j=v_{s_1}v_{s_2}^{-1}+\zeta^j+\zeta^{-j}+v_{s_1}^{-1}v_{s_2}$
and $1\leq j \leq (m-2)/2$ (if $m$ is even) or $1 \leq j \leq (m-1)/2$ (if
$m$ is odd). Note that the coefficients of the representing matrices lie 
in the ring $\Z_W[v_{s_1}^{\pm 1}, v_{s_2}^{\pm 1}]$. Now let
\[ \Omega_j=\left(\begin{array}{cc} v_{s_1}\mu_j (v_{s_2}+v_{s_2}^{-1})
& v_{s_1} \mu_j \\v_{s_1} \mu_j &v_{s_1}(v_{s_2}+v_{s_1}^{-1})
\end{array}\right) \in M_2(\Z_W[v_{s_1}^{\pm 1},v_{s_2}^{\pm 1}]).\]
Then $\Omega_j$ is a symmetric matrix satisfying $\Omega_j\rho_j
(T_{w^{-1}})=\rho_j(T_w)^{\text{tr}} \Omega_j$ for all $w \in W$. 
(By Remark~\ref{bal2a}, it is enough to verify this for $w\in \{s_1,s_2\}$.)
We see that
\begin{alignat*}{2}
\Omega_j &\equiv \left(\begin{array}{cc} 1 & 0 \\ 0 & 1\end{array}\right) 
\bmod \fp && \qquad \mbox{if $L(s_2)>L(s_1)>0$},\\
\Omega_j &\equiv \left(\begin{array}{cc} 2+\zeta^j+\zeta^{-j}&0 \\ 0& 1
\end{array}\right) \bmod \fp && \qquad \mbox{if $L(s_2)=L(s_1)>0$}.
\end{alignat*}
Hence, by Proposition~\ref{bal2}, $\rho_j$ is a balanced representation,
in all cases. Since the coefficients of the matrices $\rho_j(T_{s_1})$ and 
$\rho_j (T_{s_2})$ lie in $\Z_W$,  the same will be true for the matrices 
$\rho_j(T_w)$ where $w \in W$. Hence, all the corresponding leading matrix 
coefficients will also lie in $\Z_W$.
\end{exmp}

\begin{exmp} \label{h34} The argument in Example~\ref{leadingI2} can be
applied whenever the irreducible representations of $\bH_K$ are explicitly
known. 

Assume, for example, that $(W,S)$ is of type $H_3$ or $H_4$. In these
cases, all elements in $S$ are conjugate and so all $v_s$ ($s \in S$)
are equal; write $v=v_s$ for $s \in S$. We have $\Z_W={\Z}[\alpha]$ where 
$\alpha=\frac{1}{2}(-1+\sqrt{5})$. The irreducible representations of
$\bH_K$ are constructed by Lusztig \cite[\S 5]{Lu0} and Alvis--Lusztig 
\cite{AlLu82}, in terms of so-called {\em $W$-graphs}. (These $W$-graphs 
are reproduced in \cite[Chap.~11]{gepf}.) Thus, we obtain explicit matrix 
representations $\rho^\lambda \colon \bH_K \rightarrow M_{d_\lambda}(K)$ 
for all $\lambda \in \Lambda$. By inspection, one sees that 
\[ \rho^\lambda(T_w) \in M_{d_\lambda}(\Z_W[v,v^{-1}]) \qquad \mbox{for 
all $w \in W$}.\]
For each $\lambda \in \Lambda$, we can work out a non-zero matrix 
$\Omega^\lambda \in M_{d_\lambda}(\Z_W[v,v^{-1}])$ such that $\Omega^\lambda
\rho^\lambda(T_{w^{-1}})=\rho^\lambda(T_w)^{\text{tr}}\Omega^\lambda$ for 
all $w \in W$. (For example, with the help of a computer, we can simply 
compute $\Omega^\lambda:=\sum_{w \in W} \rho^\lambda(T_w)^{\text{tr}} 
\rho^\lambda(T_w)$.) Multiplying $\Omega^\lambda$ by a suitable scalar,
we may assume that all coefficients lie in $\Z_W[v]$ and at least one 
coefficient does not lie in $v\Z_W[v]$. For type $H_3$, the matrices 
$\Omega^\lambda$ are printed in Table~\ref{tabh3}, where we use the 
labelling of $\Irr(W)$ as in \cite[Table~C.1]{gepf}. In this case, we 
notice that the diagonal coefficients lie in $1+\fp$ while the off-diagonal
coefficients lie in $\fp$. Hence, clearly, we have $\det(\Omega^\lambda)
\in 1+\fp$. The situation in type $H_4$ is slightly more complicated, but 
one can check again that $\det(\Omega^\lambda)\in \cO^\times$ for all 
$\lambda\in\Lambda$. Thus, by Proposition~\ref{bal2}, the representations 
given by the $W$-graphs are balanced.

One may conjecture that every representation given by a $W$-graph is 
balanced.
\end{exmp}

\begin{sidewaystable}[htbp] \caption{Invariant bilinear forms for $H_3$}
\label{tabh3}
{\begin{gather*} 
\Omega^{1_r'}= \begin{bmatrix} 1 \end{bmatrix}, \qquad \Omega^{1_r}=
\begin{bmatrix} 1 \end{bmatrix}, \qquad \Omega^{5_r'}=
 \begin{bmatrix} v^2 {+} 1& 0& {-}v& 0& 0 \\ 0& v^2 {+} 1& 0& {-}v& 0\\ 
 {-}v& 0& v^2 {+} 1& 0& {-}v \\ 0& {-}v& 0& v^2 {+} 1& {-}v \\
 0& 0& {-}v& {-}v& v^2 {+} 1 \end{bmatrix},\\
\Omega^{5_r}{=}
 \begin{bmatrix} v^8 {+} v^6 {+} v^4 {+} v^2 {+} 1& v^4& v^7 {+} v^5 {+} 
 v^3 {+} v& v^5 {+} v^3& v^6 {+} v^4 {+} v^2 \\ 
 v^4& v^8 {+} v^6 {+} v^4 {+} v^2 {+} 1& v^5 {+} v^3& v^7 {+} v^5 {+} v^3 
 {+} v& v^6 {+} v^4 {+} v^2 \\ 
 v^7 {+} v^5 {+} v^3 {+} v& v^5 {+} v^3& v^8 {+} 2v^6 {+} 2v^4 {+} 2v^2 
 {+} 1& v^6 {+} 2v^4 {+} v^2& v^7 {+} 2v^5 {+} 2v^3 {+} v \\ 
 v^5 {+} v^3& v^7 {+} v^5 {+} v^3 {+} v& v^6 {+} 2v^4 {+} v^2& v^8 {+} 
 2v^6 {+} 2v^4 {+} 2v^2 {+} 1& v^7 {+} 2v^5 {+} 2v^3 {+} v \\ v^6 {+} v^4 
 {+} v^2& v^6 {+} v^4 {+} v^2& v^7 {+} 2v^5 {+} 2v^3 {+} v& v^7 {+} 2v^5 
 {+} 2v^3 {+} v& v^8 {+} 2v^6 {+} 3v^4 {+} 2v^2 {+} 1 \end{bmatrix}, \\
\Omega^{3_s}=
 \begin{bmatrix} v^2 {+} 1& {-}v& 0 \\  {-}v& v^2 {+} 1& \bar{\alpha} v 
 \\ 0& \bar{\alpha}v& v^2 {+} 1 \end{bmatrix},  \qquad 
\Omega^{3_s'}=\begin{bmatrix} v^4 {-} \alpha v^2 {+} 1& v^3 {+} v& 
 {-}\bar{\alpha} v^2 \\ v^3 {+} v& v^4 {+} 2v^2 {+} 1& 
 {-}\bar{\alpha}v^3 {-} \bar{\alpha}v \\ {-}\bar{\alpha}v^2
 & {-}\bar{\alpha}v^3 {-} \bar{\alpha}v& v^4 {+} v^2 {+} 1 
 \end{bmatrix},\\
\Omega^{\overline{3}_s}=
 \begin{bmatrix} v^2 {+} 1& {-}v& 0 \\  {-}v& v^2 {+} 1& \alpha v \\ 
 0& \alpha v& v^2 {+} 1 \end{bmatrix},\qquad \Omega^{\overline{3}_s'}=
 \begin{bmatrix} v^4 {-} \bar{\alpha}v^2 {+} 1& v^3 {+} v& 
{-}\alpha v^2 \\ v^3 {+} v& v^4 {+} 2v^2 {+} 1& 
 {-}\alpha v^3 {-} \alpha v \\ {-}\alpha v^2& {-}\alpha v^3 {-}\alpha v& 
v^4 {+} v^2 {+} 1 \end{bmatrix},\\
\Omega^{4_r'}=
 \begin{bmatrix} v^4 {-} v^3 {+} 2v^2 {-} v {+} 1& v^3 {+} v& v^3 {+} v& 
 v^2 \\ v^3 {+} v& v^4 {+} 2v^2 {+} 1& v^3 {+} v^2 {+} v& v^3 {+} v \\ 
 v^3 {+} v& v^3 {+} v^2 {+} v& v^4 {+} 2v^2 {+} 1& v^3 {+} v \\ v^2& v^3 
 {+} v& v^3 {+} v& v^4 {-} v^3 {+} 2v^2 {-} v {+} 1 \end{bmatrix},\\
\Omega^{4_r}=
 \begin{bmatrix} v^4 {+} v^3 {+} 2v^2 {+} v {+} 1& {-}v^3 {-} v& {-}v^3 {-} 
 v& v^2 \\ {-}v^3 {-} v& v^4 {+} 2v^2 {+} 1& {-}v^3 {+} v^2 {-} v& {-}v^3 
 {-} v \\ {-}v^3 {-} v& {-}v^3 {+} v^2 {-} v& v^4 {+} 2v^2 {+} 1& {-}v^3 
 {-} v \\ v^2& {-}v^3 {-} v& {-}v^3 {-} v& v^4 {+} v^3 {+} 2v^2 {+} v {+} 1 
 \end{bmatrix} 
\end{gather*}}
\end{sidewaystable}

\begin{exmp} \label{balBn} 
Let $W=W_n$ be a Coxeter group of type $B_n$, with generators $s_0,s_1,
\ldots,s_{n-1}$ and relations given by the diagram below; the 
``weights'' $a,b\in \Gamma$ attached to the generators of $W_n$ 
uniquely determine a weight function $L=L_{a,b}$ on $W_n$.
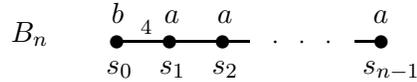
\begin{center}
\makeatletter
\vbox{\begin{picture}(200,40)
\put( 10,20){$B_n$}
\put( 50,20){\@dot{5}}
\put( 48,27){$b$}
\put( 50,20){\line(1,0){20}}
\put( 59,23){$\scriptstyle{4}$}
\put( 70,20){\@dot{5}}
\put( 68,27){$a$}
\put( 70,20){\line(1,0){30}}
\put( 90,20){\@dot{5}}
\put( 88,27){$a$}
\put(110,20){\@dot{1}}
\put(120,20){\@dot{1}}
\put(130,20){\@dot{1}}
\put(140,20){\line(1,0){10}}
\put(150,20){\@dot{5}}
\put(147,27){$a$}
\put( 46, 8){$s_0$}
\put( 66, 8){$s_1$}
\put( 86, 8){$s_2$}
\put(143, 8){$s_{n{-}1}$}
\end{picture}
\makeatother}
\end{center}
Assume that $a>0$. Then we claim that, for each $\lambda \in \Lambda$, 
there is a balanced representation $\rho^\lambda$ with corresponding 
matrix $\Omega^\lambda$ (as in Proposition~\ref{bal2}) such that 
\begin{itemize}
\item[(a)] all the leading matrix coefficients $c_{w,\lambda}^{ij}$ lie 
in $\Z$;
\item[(b)] $\Omega^\lambda\in M_{d_\lambda}(\Z[\Gamma])$ and $\det(
\Omega^\lambda)\in 2^{n_\lambda}+\fp$ where $n_\lambda \in \Z$;
\item[(c)] $n_\lambda=0$ if $b\not\in \{a,2a, \dots,(n-1)a\}$.
\end{itemize}
This can be seen by an argument which is a variation of that in 
\cite[Exp.~3.6]{geia06}. Indeed, it is well-known that we can take for
$\Lambda$ the set of all pairs of partitions of total size~$n$. Furthermore, 
for each $\lambda \in \Lambda$, we have a corresponding Specht module 
$\tilde{S}^\lambda$ as constructed by Dipper--James--Murphy \cite{DJM}. Let 
$\{e_t \mid t \in \T_\lambda\}$ be the standard basis of $\tilde{S}^\lambda$,
where $\T_\lambda$ is the set of all standard bitableaux of shape 
$\lambda$. With respect to this basis, each $T_w$ ($w \in W_n$) is 
represented by a matrix with coefficients in ${\Z}[\Gamma]$. 

Let $\langle \;, \; \rangle_\lambda$ be the invariant bilinear form on 
$\tilde{S}^\lambda$ as constructed in \cite[\S 5]{DJM}. Let $\Psi^\lambda$ 
be the Gram matrix of this bilinear form with respect to the basis 
$\{e_t\mid t \in \T_\lambda\}$. All coefficients of $\Psi^\lambda$ lie
in ${\Z}[\Gamma]$. Let $\{f_t \mid t \in \T_\lambda\}$ be the orthogonal 
basis constructed in \cite[Theorem~8.11]{DJM}; this basis is obtained from 
the standard basis by a unitriangular transformation. Hence, we have
\[ \det(\Psi^\lambda)=\prod_{t \in \T_\lambda} \langle f_t,f_t
\rangle_\lambda \in {\Z}[\Gamma].\]
Using the recursion formula in \cite[Prop.~3.8]{djm2}, it is straightforward
to show that, for each basis element $f_t$, there exist integers $s_t, 
a_{ti}, b_{tj},c_{tk}, d_{tl} \in \Z$ such that $a_{ti}\geq 0$, $b_{tj}
\geq 0$, and
\[ \langle f_t,f_t\rangle_\lambda=\varepsilon^{2s_ta}\cdot \frac{\prod_i
(1+\varepsilon^{2a}+ \cdots +\varepsilon^{2a_{ti}a})}{\prod_j (1+
\varepsilon^{2a}+\cdots +\varepsilon^{2b_{tj}a})} \cdot \frac{\prod_k 
\bigl(1+\varepsilon^{2(b+c_{tk}a)}\bigr)}{\prod_l\bigl(1+
\varepsilon^{2(b+d_{tl}a)}\bigr)}.\]
So there exist $h_t,h_t', m_{tk},m_{tl}',n_t,n_t'\in \Z$ such that
\[\begin{array}{lccl}
\prod_k \bigl(1+\varepsilon^{2(b+c_{tk}a)}\bigr)&=&2^{n_t}\,
\varepsilon^{2h_t}\prod_k \bigl(1+\varepsilon^{2m_{tk}}
\bigr) &\qquad \mbox{where $m_{tk}>0$},\\
\prod_l \bigl(1+\varepsilon^{2(b+d_{tl}a)}\bigr)&=&2^{n_t'}\,
\varepsilon^{2h_t'} \prod_l \bigl(1+\varepsilon^{2m_{tl}'} \bigr)
&\qquad \mbox{where $m_{tl}'>0$}.
\end{array}\]
Hence, setting $\tilde{e}_t:=\varepsilon^{-s_ta-h_t+h_t'}\,e_t$ and 
$\tilde{f}_t:=\varepsilon^{-s_ta-h_t+h_t'} \,f_t$, we obtain 
$2^{n_t'-n_t}\,\langle \tilde{f}_t,\tilde{f}_t \rangle_\lambda \in 1+\fp$ 
for all $t\in \T_\lambda$. Now let $\rho^\lambda$ be the matrix 
representation afforded by $\tilde{S}^\lambda$ with respect 
to $\{\tilde{e}_t \mid t \in \T_\lambda\}$ and $\Omega^\lambda$ be the 
Gram matrix of $\langle \;,\;\rangle_\lambda$ with respect to that basis. 
Then 
\[ \det(\Omega^\lambda)=\det (\Psi^\lambda)\prod_{t\in \T_\lambda} 
\varepsilon^{2(-s_ta-h_t+h_t')}=\prod_{t\in \T_\lambda} 
\bigl(\varepsilon^{2(-s_ta-h_t+h_t')}\langle f_t,f_t\rangle_\lambda\bigr)
=\prod_{t\in \T_\lambda}\langle \tilde{f}_t,\tilde{f}_t\rangle_\lambda.\]
Hence we can deduce that (a) and (b) hold. Finally, the cases in (c) 
correspond to the situations already considered in \cite[Exp.~3.6]{geia06} 
and \cite[Prop.~2.3]{BGIL}; the special feature of these cases is that 
$n_t=0$ for all $t$.
\end{exmp}

\begin{defn} \label{Lgood} Recall that $\Z_W=\Z[2\cos(2\pi/m_{st}) \mid
s,t \in S]$. We say that the subring $R\subseteq \C$ is {\em $L$-good} if 
the following conditions hold:
\begin{itemize}
\item $\Z_W\subseteq R$ and
\item $f_\lambda$ is contained and invertible in $R$, for all
$\lambda \in \Lambda$.
\end{itemize}
By Remark~\ref{ainv}, this notion does not depend on the choice of the
monomial order on $\Gamma$. Note that, if $W$ is a finite Weyl group, i.e., 
we have $m_{st} \in \{2,3,4,6\}$, then $2\cos(2\pi/m_{st})\in \Z$ and 
$f_\lambda \in \Z$ for all $\lambda \in \Lambda$. Hence, in this case, 
$\Z_W=\Z$ and the only condition on $R$ is that the integer $f_\lambda$ is 
invertible in $R$ for every $\lambda \in\Lambda$ (which is precisely the 
condition used in \cite[\S 2.2]{mycell}). 
\end{defn}

\begin{exmp} \label{bad34}
Assume that $(W,S)$ is of type $I_2(m)$ where $m=5$ or $m\geq 7$. 
Formulas for the elements $\bc_\lambda$ can be found in 
\cite[Theorem~8.3.4]{gepf}. Using these formulas, one checks that $R$ is 
$L$-good if and only if $2\cos(2\pi/m) \in R$ and the integer $m$ is 
invertible in $R$. 

Assume that $(W,S)$ is of type $H_3$. Then \cite[Table~E.2]{gepf}
shows that $R$ is $L$-good if and only if $\frac{1}{2}(1+\sqrt{5}) \in R$
and the integers $2$, $5$ are invertible in $R$.

Assume that $(W,S)$ is of type $H_4$. Then \cite[Table~E.3]{gepf}
shows that $R$ is $L$-good if and only if $\frac{1}{2}(1+\sqrt{5}) \in R$
and the integers $2,3,5$ are invertible in $R$.
\end{exmp}

\begin{prop} \label{lem1} Let $R \subseteq \C$ be a subring which is 
$L$-good. Let $\lambda \in\Lambda$.  Then the balanced representation 
$\rho^\lambda$ can be chosen such that the following hold.
\begin{itemize}
\item[(a)] $\bar{\rho}^\lambda_{ij}(t_w)=c_{w,\lambda}^{ij} \in \Z_W$ for all 
$w\in W$ and $1 \leq i,j \leq d_\lambda$.
\end{itemize}
In particular, we have $\tilde{\gamma}_{x,y,z} \in R$ for all $x,y,z\in R$.
Furthermore, there exists a symmetric, positive-definite matrix
\[B^\lambda=(\beta_{ij}^\lambda)_{1\leq i,j\leq d_\lambda} \qquad
\mbox{where} \qquad\beta_{ij}^\lambda\in \Z_W \mbox{ for all $1\leq i,j
\leq d_\lambda$},\]
such that the following two conditions hold:
\begin{itemize}
\item[(b)] $B^\lambda\,\bar{\rho}^\lambda(t_{w^{-1}})=\bar{\rho}^\lambda
(t_w)^{\operatorname{tr}}\, B^\lambda$ for all $w \in W$;
\item[(c)] $\det(B^\lambda)\neq 0$ is invertible in $R$.
\end{itemize}
\end{prop}

\begin{proof} By standard reduction arguments, one can assume that
$(W,S)$ is irreducible.  

Now (a) holds in all cases by Examples~\ref{leadp115}, \ref{leadingI2} and 
\ref{balBn}. Once this is proved, we see (by the defining formula) that 
$\tilde{\gamma}_{x,y,z}\in R$ for all $x,y,z\in W$. We can now actually take
$R$ to be the ring generated by $\Z_W$ and $f_\lambda^{-1}$ ($\lambda \in 
\Lambda$). Notice that, if $\Z_W$ is a principal ideal domain, then so 
is $R$.

Now (b) and (c) can be proved as in \cite[Prop.~2.6]{mycell}, if $\Z_W$ is 
a principal ideal domain. (In the last step of [{\em loc.\ cit.}], instead
of reducing modulo a prime number, one reduces modulo a prime ideal in $R$.)
Hence, it only remains to prove (b) and (c) for $(W,S)$ of type $I_2(m)$ 
($m \geq 3$). Note that the assertions are clear for $1$-dimensional 
representations, where we can just take $\Omega^\lambda=(1)$. For a 
$2$-dimensional representation $\rho_j$, let $\Omega_j$ be as in 
Example~\ref{leadingI2}. Let $B_j$ be the matrix obtained by taking the 
constant terms of the entries of $\Omega_j$. We notice that all entries of 
$B_j$ lie in $\Z_W$, and $B_j$ satisfies (b). It remains to consider 
$\det(B_j)$. By Example~\ref{bad34}, $m$ is 
invertible in $R$, so it will be enough to show that $\det(B_j)$ 
divides $m$ in $R$. Now, if $L(s_2)>L(s_1)> 0$, then $\det(B^\lambda)=1$ 
and so there is nothing to prove. If $L(s_1)= L(s_2)>0$, then $\det(B_j)=
2+\zeta^j+\zeta^{-j}$. Now, we have
\begin{align*}
\prod_{1\leq j\leq (m-1)/2} (2+\zeta^j+\zeta^{-j})& =1 \qquad
\mbox{if $m$ is odd},\\
\prod_{1\leq j\leq (m-2)/2} (2+\zeta^j+\zeta^{-j})& =\frac{m}{2}\qquad
\mbox{if $m$ is even}.
\end{align*}
Thus, $\det(B_j)$ divides $m$, as required. It follows that (c) holds.
\end{proof}

\begin{cor} \label{lem1a} Let $\Q_{(2)}$ be the ring of all rational numbers
of the form $2^ab$ where $a,b \in \Z$. Then $\tilde{\gamma}_{x,y,z}\in 
\Q_{(2)}$ for all $x,y,z\in W$. 
\end{cor}

\begin{proof} By standard reduction arguments, we can assume that $(W,S)$
is irreducible. Now, if {\bf P1}--{\bf P15} hold, then $\tilde{\gamma}_{x,y,
z}=\gamma_{x,y,z}\in\Z$ for all $x,y,z\in W$; see Proposition~\ref{klrem1}.
Hence, by Remark~\ref{note2}, the assertion holds in the equal parameter
case. By \cite[\S 5]{klremarks}, this also applies to $(W,S)$ of type
$F_4$ and $I_2(m)$ (for all choices of weight functions and monomial
orders). If $(W,S)$ if of type $B_n$, the result is covered by 
Example~\ref{balBn}.
\end{proof}

\section{Cellular bases} \label{seccellular}

We are now ready to review the construction of a cellular basis of $\bH$
and to extend this construction to further types of examples. We refer to 
\cite[Chap.~8]{Lusztig03} for the definition of the Kazhdan--Lusztig preorder
relation $\leq_{\cLR}$. (Note that this depends on the weight function $L$ 
and the monomial order on $\Gamma$.) For any $w \in W$, we have 
$\bH \bC_w \bH\subseteq \sum_{y} A\bC_y$ where the sum runs over all
$y \in W$ such that $y \leq_{\cLR} w$. Let $\sim_{\cLR}$ be the associated 
equivalence relation; the equivalence classes are called the two-sided cells 
of $W$. Instead of Lusztig's {\bf P1}--{\bf P15} (see 
\cite[14.2]{Lusztig03}), we shall only have to consider the following 
property which is a variant of {\bf P15}.

\medskip
\begin{center}
\fbox{$\;$ {\bf $\widetilde{\mbox{P15}}$.}
{\em If $x,x',y,w\in W$ satisfy $w \sim_{\cLR} y$, then $\;\displaystyle
\sum_{u \in W} \tilde{\gamma}_{w,x',u^{-1}}\, h_{x,u,y} =\sum_{u\in W} 
h_{x,w,u}\, \tilde{\gamma}_{u,x',y^{-1}}.\;$}}
\end{center}

\begin{rem} \label{tildeP15} Assume that {\bf P1}--{\bf P15} in
\cite[14.2]{Lusztig03} hold. Then $\tilde{\gamma}_{x,y,z}=\gamma_{x,y,z}$ 
for all $x,y,z\in W$; see Proposition~\ref{klrem1}. Now, if $x,x',y,w\in W$ 
satisfy $w \sim_{\cLR} y$, then $\ba(w)=\ba(y)$ by {\bf P4} and, hence, 
{\bf $\widetilde{\mbox{P15}}$} follows from \cite[18.9(b)]{Lusztig03}, 
which itself is deduced from {\bf P15}. Thus, 
{\bf $\widetilde{\mbox{P15}}$} holds if {\bf P1}--{\bf P15} hold.
\end{rem}

Assume from now on that $R$ is $L$-good; see Definition~\ref{Lgood}. By
Proposition~\ref{lem1}, all structure constants $\tilde{\gamma}_{x,y,z}$ 
lie in $R$. Let $\tilde{\bJ}_R$ be the $R$-span of $\{t_w \mid w \in W\}$.
Then $\tilde{\bJ}_R$ is an $R$-subalgebra of $\tilde{\bJ}$ and $\tilde{\bJ}
=F \otimes_R \tilde{\bJ}_R$. 
By the identification $\bC_w \leftrightarrow t_w$, the natural left 
$\bH$-module structure on $\bH$ (given by left multiplication) can be 
transported to a left $\bH$-module structure on $\tilde{\bJ}_A:=A
\otimes_{R} \tilde{\bJ}_R$. Explicitly, the action is given by 
\[ \bC_x.t_y=\sum_{z \in W} h_{x,y,z}\, t_z \qquad \mbox{for
all $x,y\in W$}.\]
Now we have the following result which was first proved 
by Lusztig \cite{Lu2} in the equal parameter case and in \cite[18.9 and
18.10]{Lusztig03} in general, assuming that {\bf P1}--{\bf P15} hold. Note 
that our proof is much less ``computational'' than that in 
[{\em loc.\ cit.}]; it is inspired by an analogous argument in \cite{Lu0}.

\begin{thm}[Lusztig] \label{thmJ} Assume that {\bf $\widetilde{\mbox{P15}}$} 
holds. Then there is a unique unital $A$-algebra homomorphism $\phi \colon \bH
\rightarrow \tilde{\bJ}_A$ such that, for any $h\in \bH$ and $w \in W$,
the difference $\phi(h)t_w-h.t_w$ is an $A$-linear combination of terms 
$t_y$ where $y \leq_{\cLR} w$ and $y \not\sim_{\cLR} w$. Explicitly, 
$\phi$ is given by 
\[\phi(\bC_w)=\sum_{\atop{z \in W, d \in \tilde{\cD}}{z \sim_{\cLR} d}}
h_{w,d,z} \, \tilde{n}_d\, t_d \qquad (w \in W).\]
\end{thm}

\begin{proof} Using the preorder $\leq_{\cLR}$, we can 
define a left $\bH$-module structure on $\tilde{\bJ}_A$ by the formula
\[ \bC_x \diamond t_y=\sum_{z \in W\,:\,z \sim_{\cLR} y} 
h_{x,y,z}\, t_z \qquad \mbox{for all $x,y\in W$}.\]
(More formally, one considers a graded module $\mbox{gr}(E)$ with canonical 
basis $\{\bar{e}_w\mid w \in W\}$ as in \cite[p.~492]{Lu0}, and then 
transports the structure to $\tilde{\bJ}_A$ via the identification 
$\bar{e}_w \leftrightarrow t_w$. This immediately yields the above formula.
Of course, one can also check directly that the above formula defines
a left $\bH$-module structure on $\tilde{\bJ}$.) For any $h \in \bH$ and 
$w \in W$, the difference $h.t_w-h\diamond t_w$ is an $A$-linear combination 
of terms $t_y$ where $y \leq_{\cLR} w$ and $y \not\sim_{\cLR} w$.

On the other hand, we have a natural right $\tilde{\bJ}_A$-module structure
on $\tilde{\bJ}_A$ (given by right multiplication). Then 
{\bf $\widetilde{\mbox{P15}}$} is equivalent to the statement that 
$\tilde{\bJ}_A$ is an $(\bH, \tilde{\bJ}_A)$-bimodule. Indeed, just notice 
that {\bf $\widetilde{\mbox{P15}}$} is obtained by writing out the identity 
$(\bC_x \diamond t_w)t_{x'}=\bC_x \diamond t_wt_{x'}$, where we use that, 
on both sides of {\bf $\widetilde{\mbox{P15}}$}, the sum needs only be 
extended over all $u \in W$ such that $u\sim_{\cLR} w$. (This follows from 
the fact that each $L$-block is contained in a two-sided cell; see 
Remark~\ref{preordL}(b).)

Now we can argue as follows. The left $\bH$-module structure on 
$\tilde{\bJ}_A$ gives rise to an $A$-algebra homomorphism 
\[\psi\colon  \bH \rightarrow \mbox{End}_A(\tilde{\bJ}_A) \quad
\mbox{such that}\quad \psi(h)(t_w)=h \diamond t_w.\]
Since the left action of $\bH$ on $\tilde{\bJ}_A$ commutes with the
right action of $\tilde{\bJ}_A$, the image of $\psi$ lies in
$\mbox{End}_{\tilde{\bJ}_A}(\tilde{\bJ}_A)$. Now, we have a natural
$A$-algebra isomorphism
\[ \eta \colon \mbox{End}_{\tilde{\bJ}_A}(\tilde{\bJ}_A) \rightarrow
\tilde{\bJ}_A, \qquad f \mapsto f(1_{\tilde{\bJ}_A}).\]
(This works for any ring with identity.) We define $\phi=\eta \circ
\psi\colon \bH \rightarrow {\tilde{\bJ}}_A$. Then $\phi$ is an $A$-algebra 
homomorphism such that
\[ \phi(h)=\psi(h)(1_{{\tilde{\bJ}}_A})=h \diamond 1_{{\tilde{\bJ}}_A}
\qquad \mbox{for all $h \in \bH$}.\]
This yields $\phi(h)t_w=(h \diamond 1_{\tilde{\bJ}_A})t_w=
h \diamond 1_{\tilde{\bJ}_A}t_w=h \diamond t_w$ or, in other words, the 
difference $\phi(h)t_w-h.t_w$ is an $A$-linear combination of terms $t_y$ 
where $y \leq_{\cLR} w$ and $y \not\sim_{\cLR} w$, as required. 
Finally, we immediately obtain the formula
\[\phi(\bC_w)=\bC_w \diamond 1_{\tilde{\bJ}_A}=\sum_{d \in \tilde{\cD}} 
\tilde{n}_d \, \bC_w \diamond t_d=\sum_{\atop{z \in W, d \in \tilde{\cD}}{z 
\sim_{\cLR} d}} h_{w,d,z} \, \tilde{n}_d\, t_z.\]
Since $h_{1,d,z}=\delta_{d,z}$, this yields $\phi(\bC_1)=
1_{\tilde{\bJ}_A}$, hence $\phi$ is unital. 

The unicity of $\phi$ is clear since the conditions on $\phi$ imply
that $\phi(h)t_w=h\diamond t_w$ for all $w \in W$ and, hence, $\phi(h)
=\phi(h)1_{\tilde{\bJ}_A}=h \diamond 1_{\tilde{\bJ}_A}$ for all $h \in \bH$.
 \end{proof}

\begin{rem} \label{note1}  Assume that {\bf P1}--{\bf P15} hold. Then
$\tilde{\gamma}_{x,y,z}=\gamma_{x,y,z}$ for all $x,y,z\in W$; see 
Proposition~\ref{klrem1}. Hence, $\tilde{\bJ}$ is Lusztig's ring $\bJ$
constructed in \cite[Chap.~18]{Lusztig03}. Since the identity element is 
uniquely determined, we can also conclude that $\tilde{\cD}=\cD$ and 
$\tilde{n}_d=n_d$ for all $d \in \cD$, where $\cD$ and $n_d$ are defined as 
in [{\em loc.\ cit.}]. Hence, the above result is a combination of 
\cite[Theorems~18.9 and 18.10]{Lusztig03}. 

Note that the formula for $\phi$ in \cite[18.9]{Lusztig03} looks somewhat 
different: there is a factor $\hat{n}_z$ instead of $\tilde{n}_d=n_d$.
However, by \cite[Rem.~2.10]{klremarks}, one can easily see
that the two versions are equivalent. And in view of the above proof,
the version here seems more natural.
\end{rem}

Finally, we come to the construction of ``cell data'' for $\bH$ in the 
sense of Graham and Lehrer \cite{GrLe}. By \cite[Definition~1.1]{GrLe}, we 
must specify a quadruple $(\Lambda,M,C,*)$ satisfying the following 
conditions.
\begin{itemize}
\item[(C1)] $\Lambda$ is a partially ordered set (with partial order
denoted by $\trianglelefteq$), $\{M(\lambda) \mid \lambda \in \Lambda\}$ 
is a collection of finite sets  and
\[ C \colon \coprod_{\lambda \in \Lambda} M(\lambda)\times M(\lambda)
\rightarrow \bH \]
is an injective map whose image is an $A$-basis of $\bH$;
\item[(C2)] If $\lambda \in \Lambda$ and $\fs,\ft\in M(\lambda)$, write
$C(\fs,\ft)=C_{\fs,\ft}^\lambda \in \bH$. Then $* \colon \bH \rightarrow 
\bH$ is an $A$-linear anti-involution such that $(C_{\fs,\ft}^\lambda)^*=
C_{\ft,\fs}^\lambda$.
\item[(C3)] If $\lambda \in \Lambda$ and $\fs,\ft\in M(\lambda)$, then for
any element $h \in \bH$ we have
\[ hC_{\fs,\ft}^\lambda\equiv \sum_{\fs'\in M(\lambda)} r_h(\fs',\fs)\,
C_{\fs',\ft}^\lambda\quad \bmod \bH(\triangleleft\,\lambda),\]
where $r_h(\fs',\fs) \in A$ is independent of $\ft$ and where 
$\bH(\triangleleft\,\lambda)$ is the $A$-submodule of $\bH$ generated by 
$\{C_{\fs'',\ft''}^\mu \mid \mu \trianglelefteq \lambda; \lambda \neq 
\mu; \fs'',\ft''\in M(\mu)\}$.
\end{itemize}
We now define a required quadruple $(\Lambda,M,C,*)$ as follows.

As before, $\Lambda$ is an indexing set for the irreducible representations 
of $W$. For $\lambda\in \Lambda$, we set $M(\lambda)=\{1,\ldots,d_\lambda\}$.
We define a partial order on $\Lambda$ as follows. Recall that, in 
Remark~\ref{preordL}, we have associated with $\lambda \in \Lambda$ an
``$L$-block'' $\fF_\lambda$ of $W$. Now, given $\lambda, \mu \in \Lambda$, 
let $x \in \fF_\lambda$ and $y \in \fF_\mu$; then we define 
\begin{center}
\fbox{$\;\lambda \trianglelefteq\mu \qquad 
\stackrel{\text{def}}{\Leftrightarrow} \qquad \lambda=\mu \quad \mbox{ or } 
\quad x \leq_{\cLR} y,\;\; x\not\sim_{\cLR} y.\;$}
\end{center}
(This does not depend on the choice of $x$ or $y$, since each $L$-block
is contained in a two-sided cell of $W$; see Remark~\ref{preordL}(b).)

\begin{rem} \label{preordrem} Assume that {\bf P1}--{\bf P15} in
\cite[14.2]{Lusztig03} hold. By Proposition~\ref{klrem1}, we then have 
$\ba(z)=\ba_\lambda$ if $\bar{\rho}^\lambda(t_z)\neq 0$. Furthermore, by 
{\bf P4} and {\bf P11}, we have the implication ``$x\leq_{\cLR} y
\Rightarrow \ba(y)\leq \ba(x)$'', with equality only if $x \sim_{\cLR} y$. 
Hence, we see that
\[ \lambda \trianglelefteq\mu \qquad \Rightarrow \qquad \lambda=\mu \quad 
\mbox{or} \quad \ba_\mu<\ba_\lambda.\]
The partial order defined by the condition on the right hand side is the
one we used in \cite{mycell}.
\end{rem}
Finally, we define an $A$-linear anti-involution $* \colon \bH \rightarrow 
\bH$ by $T_w^*=T_{w^{-1}}$ for all $w \in W$. Thus, $T_w^*=T_w^{\,\flat}$ 
in the notation of \cite[3.4]{Lusztig03}. We can now state the following 
result:

\begin{thm}[Cf.\ \protect{\cite[Theorem~3.1]{mycell}}] \label{mainthm} 
Assume that {\bf $\widetilde{\mbox{P15}}$} holds. Recall that $R\subseteq\C$ 
is assumed to be an $L$-good subring; see Definition~\ref{Lgood}. Let 
$\bigl(\bar{\rho}_{\fs\ft}^\lambda(t_w)\bigr)$ and $\bigl(\beta_{\fs
\ft}^\lambda \bigr)$ be as in Proposition~\ref{lem1}. For any $\lambda 
\in \Lambda$ and $\fs,\ft\in M(\lambda)$, define
\[C_{\fs,\ft}^\lambda=\sum_{w \in W} \sum_{\fu \in M(\lambda)}
\beta^\lambda_{\ft\fu}\,\bar{\rho}_{\fu\fs}^\lambda\, (t_{w^{-1}})\, \bC_w.\]
Then $C_{\fs,\ft}^\lambda$ is a $\Z_W$-linear combination of Kazhdan--Lusztig
basis elements $\bC_w$ where $w\in \fF_\lambda$. The quadruple 
$(\Lambda,M,C,*)$ is a ``cell datum'' in the sense of Graham--Lehrer
\cite{GrLe}.
\end{thm}

\begin{proof} In all essential points, the argument is the same as
in the proof of \cite[Theorem~3.1]{mycell}. Indeed, since 
{\bf $\widetilde{\mbox{P15}}$} holds, we have the existence of Lusztig's 
homomorphism $\phi\colon \bH \rightarrow \tilde{\bJ}_A$ as in 
Theorem~\ref{thmJ}. The statements in Proposition~\ref{lem1} are completely 
analogous to those in \cite[Prop.~2.6]{mycell}. Finally, by 
Theorem~\ref{thmJ}, we have the property that $\phi(h)t_w-h.t_w$ is an 
$A$-linear combination of terms $t_y$ where $y \leq_{\cLR}$ and $y
\sim_{\cLR} w$. This is precisely what is needed in order to make 
Step~3 of the proof of \cite[Theorem~3.1]{mycell} work with our stronger 
definition of the partial order $\trianglelefteq$ on~$\Lambda$.
\end{proof}

The above result strengthens the main result of \cite{mycell} in four ways:
\begin{itemize}
\item it works for finite Coxeter groups in general, and not just for
Weyl groups;
\item it only requires {\bf $\widetilde{\mbox{P15}}$} to hold, and not all 
of {\bf P1}--{\bf P15} in \cite[14.2]{Lusztig03};
\item it uses a slightly stronger partial order on $\Lambda$ (see
Remark~\ref{preordrem});
\item it shows that the data required to define the cellular basis can
be extracted from the balanced representations $\rho^\lambda$.
\end{itemize}

\begin{cor} \label{cellfin} Let $(W,S)$ be any Coxeter system where
$W$ is finite. Let $R \subseteq \C$ be a subring which is $L_0$-good,
where $L_0$ is the ``univeral'' weight function in Example~\ref{Mrem12}.
Now let $L' \colon W \rightarrow \Gamma'$ be {\em any} weight function 
and $\bH'$ the corresponding Iwahori--Hecke algebra over $A'=R[\Gamma']$. 
Then $\bH'$ admits a cell datum in the sense of Graham--Lehrer \cite{GrLe}.
\end{cor}

\begin{proof} Let $\Gamma_{0}$, $A_0$ and $\bH_0$ be as in 
Example~\ref{Mrem12}. As pointed out in \cite[Cor.~5.4]{klremarks}, by 
combining all the known results about the validity of Lusztig's conjectures 
\cite[14.2]{Lusztig03}, we can choose a monomial order $\leq$ on $\Gamma_0$ 
such that {\bf P1}--{\bf P15} hold. Hence, by Remark~\ref{tildeP15} and
Theorem~\ref{mainthm}, the algebra $\bH_0$ admits a cell datum.
Now, there is a group homomorphism $\alpha\colon \Gamma_0 \rightarrow 
\Gamma'$ such that $\alpha((n_s)_{s \in S})=\sum_{s \in S} n_sL'(s)$.
This extends to a ring homomorphism $A_0 \rightarrow A'$ which we denote 
by the same symbol. Extending scalars from $A_0$ to $A'$ (via $\alpha$), we 
obtain $\bH'=A'\otimes_{A_0} \bH_0$. By \cite[Cor.~3.2]{mycell}, the images 
of the cellular basis elements of $\bH_0$ in $\bH'$ form a cellular basis 
in $\bH'$.
\end{proof}

In type $B_n$, an alternative construction of a cell datum is 
given by Dipper--James--Murphy \cite{DJM}.

 

\begin{thebibliography}{131} 
 
\bibitem{AlLu82}
{\sc D.~Alvis and G.~Lusztig}, {The representations and generic degrees of 
the Hecke algebra of type ${H}_4$}, J. Reine Angew. Math. {\bf 336} (1982), 
201--212; correction, {\em ibid.} \textbf{449} (1994), 217--218.

\bibitem{BGIL}
{\sc  C.~Bonnaf\'e, M.~Geck, L.~Iancu and T.~Lam},  On domino
insertion and Kazhdan--Lusztig cells in type $B_n$, preprint; available
at {\tt math.RT/0609279}.

\bibitem{BI}
{\sc C.~Bonnaf\'e and L.~Iancu}, Left cells in type $B_n$ with unequal
parameters, Represent. Theory {\bf 7} (2003), 587--609.

\bibitem{BrMa93}
{\sc M.~Brou{\'e} and G.~Malle}, Zyklotomische {H}eckealgebren, 
Ast\'erisque {\bf 212} (1993), 119--189.

\bibitem{DJM}
{\sc R.~Dipper, G. D.~James and G.~E.~Murphy}, Hecke algebras of type $B_n$
at roots of unity, Proc. London Math. Soc. {\bf 70} (1995), 505--528.

\bibitem{djm2}
{\sc R.~Dipper, G.~James and G. E.~Murphy}, Gram determinants of type $B_n$,
J. Algebra {\bf 189} (1997), 481--505.

\bibitem{Fokko}
{\sc F.~DuCloux}, Positivity results for the Hecke algebras of
noncrystallographic finite Coxeter group, J. Algebra {\bf 303} (2006),
731--741.

\bibitem{my02}
{\sc M.~Geck}, Constructible characters, leading coefficients and left cells
for finite Coxeter groups with unequal parameters, Represent. Theory {\bf 6}
(2002), 1--30 (electronic).


\bibitem{mycell}
{\sc M. Geck}, Hecke algebras of finite type are cellular, Invent. Math.
{\bf 169} (2007), 501--517.

\bibitem{klremarks}
{\sc M. Geck}, Remarks on Hecke algebras with unequal parameters, preprint,
{\tt arXiv:0711.2522}.

\bibitem{geia06}
{\sc M. Geck and L.~Iancu}, Lusztig's $a$-function in type $B_n$ in
the asymptotic case. Special issue celebrating the $60$th birthday of
George Lusztig, Nagoya J. Math. {\bf 182} (2006), 199--240.

\bibitem{gepf} 
{\sc M. Geck and G. Pfeiffer}, Characters of finite Coxeter groups and 
Iwahori--Hecke algebras, London Math. Soc. Monographs, New Series {\bf 21}, 
Oxford University Press, New York 2000. xvi+446 pp. 
 
\bibitem{GrLe}
{\sc J.~J.~Graham and G.~I.~Lehrer}, Cellular algebras,
Invent.~Math. {\bf 123} (1996), 1--34.

\bibitem{KaLu}
{\sc D.~A.~Kazhdan and G.~Lusztig}, {Representations of {C}oxeter groups and
{H}ecke algebras}, Invent. Math. {\bf 53} (1979), 165--184.

\bibitem{Lu0}
{\sc G.~Lusztig}, {On a theorem of {B}enson and {C}urtis},
J. Algebra {\bf 71} (1981), 490--498.

\bibitem{Lusztig81b}
{\sc G.~Lusztig}, {Unipotent characters of the symplectic and odd orthogonal
groups over a finite field}, Invent. Math. \textbf{64} (1981), 263--296.

\bibitem{Lusztig83}
{\sc G.~Lusztig}, Left cells in {W}eyl groups, {\em Lie Group
Representations, I} (eds R.~L. R.~Herb and J.~Rosenberg), Lecture Notes
in Mathematics 1024 (Springer, Berlin, 1983), pp.~99--111.

\bibitem{LuBook}
{\sc G.~Lusztig}, Characters of reductive groups over a finite field,
Annals Math.\ Studies, vol. 107, Princeton University Press, 1984.

\bibitem{Lu2}
{\sc G. Lusztig}, Cells in affine Weyl groups II, J. Algebra {\bf 109} 
(1987), 536--548.

\bibitem{Lusztig03}
{\sc G.~Lusztig}, Hecke algebras with unequal parameters, CRM Monographs
Ser.~{\bf 18}, Amer. Math. Soc., Providence, RI, 2003.

\end{thebibliography}
\end{document}